\documentclass[%leqno,
a4paper,11pt]{article}

\usepackage{hyperref}
\usepackage{amsthm}
\usepackage{amsmath}
\usepackage{amssymb}

\usepackage{enumitem}

\usepackage{mathtools}
\mathtoolsset{showonlyrefs}

\usepackage[xcolor]{changebar}

\newcommand{\Con}{\ensuremath{\mathcal{C}}}

\newcommand{\mb}[1]{\ensuremath{\mathbb{#1}}}
\newcommand{\N}{\mb{N}}

\newcommand{\R}{\mb{R}}

\newcommand{\A}{\ensuremath{{\mathcal A}}}

\newfont{\bl}{msbm10 scaled \magstep2}

\newcommand{\beq}{\begin{equation}}
\newcommand{\eeq}{\end{equation}}
\newcommand{\notmid}{\mid\kern-0.5em\not\kern0.5em}

\newcommand{\al}{\alpha}

\newcommand{\eps}{\varepsilon}

\newcommand{\vphi}{\varphi}

\newtheorem{thm}{Theorem}[section]
\newtheorem{lem}[thm]{Lemma}
\newtheorem{prop}[thm]{Proposition}
\newtheorem{cor}[thm]{Corollary}
\newtheorem{defi}[thm]{Definition}

\theoremstyle{definition}
\newtheorem{rem}[thm]{Remark}
\newenvironment{pr}{\begin{proof}[\textbf{Proof:}] \ }{\end{proof}}
\newtheorem{ex}[thm]{Example}

\newcommand{\bx}{\bar{x}}
\newcommand{\by}{\bar{y}}
\newcommand{\bz}{\bar{z}}

\newcommand{\LLS}{Lorentzian length space }
\newcommand{\LLSn}{Lorentzian length space}
\newcommand{\LpLS}{Lorentzian pre-length space }
\newcommand{\LpLSn}{Lorentzian pre-length space}

\newcommand{\Yll}{$(Y,d,\ll,\leq,\tau)$ }

\newcommand{\Tau}{\mathcal{T}}

\newcommand{\dd}{\mathrm{d}}
\newcommand{\diam}{\mathrm{diam}}

\newcommand{\wpd}{generalized cone}
\newcommand{\wpds}{generalized cones}
\newcommand{\hyp}{\mathbb{H}}
\newcommand{\lv}{L^{\mathrm{var}}}

\newcommand{\I}{\mathcal{I}}
\newcommand{\lstr}{Lorentzian length structure }
\newcommand{\lstrn}{Lorentzian length structure}

\newcounter{desccount}

\newcommand{\descref}[1]{{\bf (\hyperref[#1]{#1})}}

\newtheorem*{Theorem48}{Theorem 4.8 \& Corollary 4.9}

\title{Generalized cones as Lorentzian length spaces: Causality, curvature, and singularity theorems}
 \author{Stephanie B. Alexander\thanks{{\tt sba@illinois.edu}, Department of Mathematics, University of Illinois at 
Urbana-Champaign, USA.}, Melanie Graf\thanks{{\tt mgraf2@uw.edu}, Department of Mathematics, University 
of Washington, USA.}, Michael Kunzinger\thanks{{\tt michael.kunzinger@univie.ac.at}, Faculty of Mathematics, University 
of Vienna, Austria.}, Clemens S\"amann\thanks{{\tt clemens.saemann@utoronto.ca}, Department of Mathematics, University of Toronto, Canada.}}

\begin{document}

 \maketitle
 
 \begin{abstract}
We study generalizations of Lorentzian warped products with one-dimensional base of the form $I\times_f X$, where $I$ is an interval, $X$ is a length space and $f$ is a
positive continuous function. These \emph{generalized cones} furnish an important
class of Lorentzian length spaces in the sense of \cite{KS:18}, displaying
optimal causality properties that allow for explicit descriptions of all
underlying notions. In addition, synthetic sectional curvature bounds of
generalized cones are directly related to metric curvature bounds of the fiber $X$.
The interest in such spaces comes both from metric geometry and from 
General Relativity, where warped products underlie important cosmological models (FLRW spacetimes). Moreover, we prove singularity theorems for these spaces, showing that non-positive lower timelike curvature bounds imply the existence 
of incomplete timelike geodesics.
\bigskip

\medskip

 
\end{abstract}

\newpage
 \tableofcontents

\section{Introduction}
Warped products are of central importance to Riemannian geometry, in particular in the study of constant curvature geometries and as a rich source of examples and counterexamples (cf., e.g., \cite{Pet:16}). Generalized cones and warped products of metric spaces likewise play an important role in the theory of length metric spaces with synthetic curvature bounds (\emph{Alexandrov spaces}). These spaces, while including Riemannian manifolds with curvature bounds, allow singularities, being closed, for example, under Gromov-Hausdorff limits or gluing operations. Alexandrov spaces have yielded major insights into classical Riemannian geometry  (\cite{Per:93}, cf.\  \cite{Kap:07, Gro:01}). Generalized cones and warped products provide examples and counterexamples in Alexandrov geometry  (cf.\ \cite{Che:99, AB:98, AB:04, AB:16}). Moreover, the first-order structure is captured by the \emph{tangent cone};  for instance, for Alexandrov spaces of curvature bounded below, the tangent cone at a point is homeomorphic to a small metric ball centered at the point (cf.\ \cite{BBI:01}).
 
In the smooth pseudo-Riemannian setting, Alexander and Bishop gave in \cite{AB:08} a characterization of sectional curvature bounds in terms of triangle comparison, including applications to Lorentzian and general semi-Riemannian warped products and so-called \emph{Friedmann-Lemaitre-Robertson-Walker (FLRW)}-spacetimes. Lorentzian geometry enjoys a unique position in the smooth pseudo-Riemannian world: Pseudo-Riemannian metrics of signature $(-,+,\dots ,+)$ are central to the study of the physical theory of General Relativity and (apart from the Riemannian case) Lorentzian theory is the most mathematically well explored case, providing many strong tools and results which are absent in more general signatures. Additionally, the past decade has seen increasing interest from the mathematical physics community in the study of low regularity Lorentzian geometry and General Relativity. An intensive study of causality theory (cf.\ \cite{Min:19b}) for Lorentzian metrics of low regularity was initiated by P.\ Chru\'sciel and J.D.E.\ Grant in \cite{CG:12} and then pursued by various authors, see \cite{Min:15,KSS:14,KSSV:14,Sae:16}. In particular, Chru\'sciel and Grant showed in \cite{CG:12} that for spacetimes with merely continuous metrics pathologies of the causal structure may occur, e.g.\ there are so-called \emph{causal bubbles}, where the boundary of the lightcone is not a hypersurface but has positive measure (see also \cite{GKSS:20}). 
An important recent result for continuous Lorentzian metrics is the $\Con^0$-inextendibility of the Schwarzschild solution to the Einstein equations, which was shown by J.\ Sbierski in \cite{Sbi:18} and has sparked further research into low regularity (in-)extendibility and causality (e.g.\ \cite{GLS:18, GL:17, DL:17, GL:18, GKS:19}). The importance of such low regularity (in-)extendibility results is rooted in the strong cosmic censorship conjecture (cf.\ e.g.\ \cite{Ise:15}), which, roughly, states that the maximal globally hyperbolic development of generic initial data for the Einstein equations is inextendible as a suitably regular Lorentzian manifold and which is intimately related to the question of determinism in General Relativity. Another area of mathematical general relativity, where low regularity has recently come to the forefront of current research, is the study of singularities and in particular of so-called singularity theorems, predicting causal geodesic incompleteness under certain curvature and causality assumptions. The classical singularity theorems of Hawking and Penrose have only recently been successfully extended to the $\Con^{1,1}$-setting (\cite{KSSV:15, KSV:15, GGKS:18}), which is a natural regularity class to consider as curvature is still almost everywhere defined and locally bounded. Furthermore, these singularity theorems have recently been extended to the regularity class $\Con^1$ in \cite{Gra:20}. With the final results in Section \ref{sec:sin-thm} this paper will directly contribute to this line of research by proving singularity theorems for generalized cones. We also mention another natural generalization of smooth Lorentzian geometry, namely cone structures on differentiable manifolds and Lorentz-Finsler spacetimes, see \cite{FS:12,BS:18,Min:19a, MS:19, LMO:19}, which proved to be a significant extension of the field. Also, there currently is strong interest in bringing techniques from Riemannian geometry and optimal transport into the Lorentzian setting, cf.\ \cite{CM:20, McC:20, MS:18}. For extending these techniques further to a synthetic setting it might prove useful to use generalized cones  as introduced and studied in the present paper as a starting point.

Finally, we note that there have been several approaches to a synthetic or axiomatic description of (parts of) Lorentzian geometry and causality. We mention in particular the \emph{timelike spaces} of Busemann \cite{Bus:67} and the \emph{causal set theory} of quantum gravity \cite{BLMS:87, Sur:19}. For a more detailed discussion see the introduction and Subsection 5.3 of \cite{KS:18}.
Another closely related direction of research is the recent approach of Sormani and Vega \cite{SV:16} and its further development by Allen and Burtscher in \cite{AB:19} of defining a metric on a spacetime that is compatible with the causal structure in case the spacetime admits a time function satisfying an anti-Lipschitz condition.

The importance of warped products, specifically, in General Relativity (cf., e.g., \cite{ONe:83,Wal:84,Min:07}) stems from the fact that the FLRW models of the universe in cosmology are particular examples of warped products with one-dimensional base, see e.g.\ \cite[Ch.\ 12]{ONe:83}. Such spaces have a very simple structure geometrically and provide a good starting point for trying to generalize the smooth theory. Our main object of study will be generalizations of Lorentzian warped products with one-dimensional base to the case where the fiber is merely a metric length space, but some of our results are new even if the fiber is a smooth Riemannian manifold. For example, we show that there is no causal bubbling in such spacetimes even if the warping function $f$, and hence the Lorentzian  metric, is merely continuous and that maximizing causal curves of positive length have to be timelike. Moreover, any globally hyperbolic smooth spacetime $(M,g)$ splits isometrically as $(M,g)\cong (\R\times S,-\beta\,  \dd t^2 + h_t)$, where $S$ is a Cauchy hypersurface in $M$, $\beta$ is a smooth positive function on $\R\times S$ and $h_t$ is a $t$-dependent family of Riemannian metrics on each level set $\{t\}\times S$ (cf.\ \cite{BS:03, BS:05}). Globally hyperbolic spacetimes can therefore be viewed as generalizations of warped products with one-dimensional base, so our methods may also find applications to such spaces in future research.

Both from the perspective of Lorentzian geometry and with a view to the fundamental importance of warped products in General Relativity it is therefore of interest to study generalizations of such geometries beyond the setting of smooth manifolds. A natural framework in which to carry out such an extension is the theory of Lorentzian length spaces (\cite{KS:18,GKS:19}), see Subsection \ref{subsec-lls} below for a brief introduction. 

\subsection{Main results and outline of the paper}
The plan of the paper is as follows. In the remainder of this introduction we recall
the basic notions of the theory of metric spaces with curvature bounds and Lorentzian length spaces. Section \ref{sec-con} introduces Minkowski-cones, a Lorentzian analogue of cones over metric spaces, and a first instance of a generalized cone as defined in Section \ref{sec:gen_cone}.
The main result of Section \ref{sec-con} relates curvature bounds of the metric space $X$ (the fiber) to timelike curvature bounds of the cone over $X$. In fact, we prove
\setcounter{section}{2}
\setcounter{thm}{4}
\begin{thm}
Let $Y=\mathrm{Cone}(X)$ be the Minkowski cone over a geodesic length space $X$. Then $Y$ has timelike curvature bounded below (above) by $0$ if and only if $X$ is an Alexandrov space of curvature bounded below (above) by $-1$.
\end{thm}

In Section \ref{sec:gen_cone}, we introduce the main object of this paper, a metric analogue of Lorentzian warped products with one-dimensional base and a length space as fiber. We then study timelike and causal curves, introduce a time-separation function and establish the main features of causality theory for these \emph{generalized cones}.
While there are a number of analogues to the metric theory of warped products (e.g., fiber independence of geodesics), these causality results require new methods. The main result in this section is that generalized cones display no causal pathologies, a fact that is used extensively later on. In more detail, we show that

\setcounter{section}{3}
\setcounter{thm}{21}
\begin{prop}(Push-up and openness of $I^\pm$)
Every \wpd\ $Y=I\times_f X$ such that $(X,d)$ is a length space has the property that $p\ll q$ if and only if there 
exists a future directed causal curve from $p$ to $q$ of positive length, i.e., push-up holds. 
Moreover, $I^\pm(p)$ is open for any  $p\in Y$. 
\end{prop} 

The above result is then used in Section \ref{sec:gen_cones_as_lls} to establish that generalized cones are examples of Lorentzian length spaces, without any additional assumptions on the causality or on the warping function $f$ (continuous and positive):

\setcounter{section}{4}
\setcounter{thm}{7}
\begin{Theorem48} \emph{
 Any \wpd\ $I\times_f X$, where $(X,d)$ is a locally compact length space, is a strongly causal \LLSn. If $X$ is a locally compact geodesic length space, then $I\times_f X$ is a regular strongly causal \LLSn. }
\end{Theorem48}
In particular we prove that if the fiber $X$ is a geodesic length space that is proper then $I\times_f X$ is globally hyperbolic. Section \ref{sec-cb} is then devoted to relating synthetic curvature bounds (via triangle comparison) in generalized cones to corresponding bounds in the fiber.
Here the main results are as follows:
\setcounter{section}{5}
\setcounter{thm}{2}
\begin{thm}
	Let $K, K'\in \R$ and let $(X,d)$ be a geodesic length space with curvature bounded below/above by $K$. 
	Then $Y=I\times_fX$ has timelike curvature bounded below/above by $K'$ if
	$I\times_f \mathbb{M}^2(K)$ has timelike curvature bounded below/above by $K'$.
\end{thm}

\setcounter{thm}{6}
\begin{thm}
	If $X$ is a geodesic length space, $Y=I\times_f X$ has timelike curvature bounded below (above) by $K'$ and $Y'=I\times_f \mathbb{M}^2(K)$ has timelike curvature bounded above (below) by $K'$ then $X$ has curvature bounded below (above) by $K$.
\end{thm}
Moreover, the first result above allows us to generate an abundance of examples of \LLSn s with timelike curvature bounds.

We then apply our techniques in Section \ref{sec:sin-thm} to show that non-positive lower timelike curvature bounds imply the existence of incomplete timelike geodesics. That is, we provide \emph{synthetic singularity theorems} for generalized cones. To be precise, we prove the following:

\setcounter{section}{6}
\setcounter{thm}{3}

\begin{cor} Let $X$ be a geodesic length space, $Y=I\times_f X$ with $I=(a,b),\,f:I\to (0,\infty)$ smooth. Assume that $Y$ has timelike curvature bounded below by $K$. Then:
	\begin{enumerate}
		\item[(i)] If $K<0$, then $a>-\infty$ and $b<\infty$ and hence the time separation function $\tau_Y $ of $Y$ is bounded by $b-a$. Thus any such $Y$ is timelike geodesically incomplete.
		\item[(ii)] If $K=0$ and $f$ is non-constant, then $a>-\infty $ 
		or $b<\infty$ 
		and hence $Y$ is past or future timelike geodesically incomplete.
	\end{enumerate}
\end{cor}
These results may be viewed as sectional curvature analogues of the Lorentzian Bonnet-Myers' theorem and of Hawking's singularity theorem in the setting of generalized cones. Also, we relate timelike curvature bounds to \emph{big bang} and \emph{big crunch} 
singularities in Corollary \ref{cor-big-ban}.

In the appendix we describe a general approach to what we call \emph{Lorentzian length structures}, analogous to the theory of length structures in metric geometry (cf.\ \cite[Ch.\ 2]{BBI:01}) based on which several basic results shown in Sections \ref{sec-con} and \ref{sec:gen_cone} can be shown to hold in greater generality.

\setcounter{section}{1}

\subsection{Alexandrov spaces}\label{subsec-ale}
We briefly recall the basic definitions of Alexandrov spaces, i.e., metric spaces with curvature bounds. For comprehensive introductions see \cite{AKP:19, BBI:01, BH:99}.

A metric space $(X,d)$ is a \emph{length space} if for all $x,y\in X$ one has $d(x,y)=\inf\{L^d(\gamma): \gamma $ continuous and connects $x,y\}$, where $L^d(\gamma)$ is the length of $\gamma$. A \emph{geodesic} in a metric space is a continuous curve $\gamma\colon[0,l]\rightarrow X$ such that $d(\gamma(t),\gamma(s))=|t-s|$ for all $t,s\in [0,l]$. A metric space is \emph{geodesic} if any two points can be joined by a geodesic.

We use $\mathbb{M}^2(K)$ to denote the Riemannian model space of constant sectional curvature $K$, i.e.,
\begin{equation}\label{eq:Riem_model_spaces}
\mathbb{M}^2(K) = \left\{ \begin{array}{ll}
\mathbb{S}^2(r) & K=\frac{1}{r^2}\\
\R^2 & K=0\\
\hyp^2(r) & K= -\frac{1}{r^2}\,.
\end{array}
\right.
\end{equation}
Moreover, a \emph{triangle} $\Delta$ in a metric space $(X,d)$ is a triple of points $\Delta=(x,y,z)$ and a choice of geodesic segments joining $x,y,z$, i.e., its \emph{sides} $[xy]$, $[xz]$ and $[yz]$. A \emph{comparison triangle} of $\Delta$ is a triangle $\bar\Delta=(\bar x,\bar y,\bar z)$ in some model space $\mathbb{M}^2(K)$ (for some $K\in\R$) that has the same side lengths as $\Delta$, i.e., $d(x,y)=\bar d(\bar x,\bar y)$, $d(x,z)=\bar d(\bar x,\bar z)$ and $d(y,z)=\bar d(\bar y,\bar z)$, where $\bar d$ is the metric on $\mathbb{M}^2(K)$.
% cf.\ \cite[Def.\ 4.6.5]{BBI:01}\\

A length space $(X,d)$ has curvature bounded below/above by $K\in\R$ if every point $x_0\in X$ has a neighborhood $U$ such that for any triangle $\Delta=(x,y,z)$ in $U$ and any point $w$ on the side $[yz]$ the following holds: Let $\bar\Delta=(\bar x,\bar y,\bar z)$ be a comparison triangle for $\Delta$ in $\mathbb{M}^2(K)$ and let $\bar w$ on the side $[\bar y\bar z]$ with the same distance to $y$ (or $z$), i.e., $d(y,w)=\bar d(\bar y,\bar w)$, where $\bar d$ is the metric on $\mathbb{M}^2(K)$. Then
\begin{equation}
 d(x,w)\geq \bar d(\bar x,\bar w)\qquad / \qquad d(x,w)\leq \bar d(\bar x,\bar w)\,.
\end{equation}

\subsection{\LLSn s}\label{subsec-lls}
Here we give a very brief introduction to the theory of \LLSn s, as developed in \cite{KS:18}, at the same time 
fixing some notations and terminology.

Let $Y$ be a set endowed with a preorder $\leq$ and a transitive relation $\ll$ contained in $\leq$. If $x\ll y$ 
or $x\le y$, we call $x$ and $y$ timelike or causally related, respectively.  If $Y$ is, in addition, equipped with a 
metric $d$ and a lower semicontinuous map $\tau \colon Y\times Y \to [0, \infty]$ that satisfies the reverse triangle 
inequality $\tau(x,z)\geq \tau(x,y) + \tau(y,z)$ (for all $x\leq y\leq z$), as well as $\tau(x,y)=0$ if $x\nleq y$ and 
$\tau(x,y)>0 \Leftrightarrow x\ll y$, then \Yll is called a \emph{Lorentzian pre-length space\/} and $\tau$ is called the 
\emph{time separation function\/} (or \emph{Lorentzian distance}) of $Y$.

A curve $\gamma \colon I\rightarrow Y$ ($I$ an interval) that is non-constant on any sub-interval of $I$ is called 
(future-directed) \emph{causal (timelike)} if $\gamma$ is locally Lipschitz continuous and if for all 
$t_1,t_2\in I$ with $t_1<t_2$ we have $\gamma(t_1)\leq\gamma(t_2)$ ($\gamma(t_1)\ll\gamma(t_2)$). It
is called \emph{null\/} if, in addition to being causal, no two points on the curve are related with respect 
to $\ll$. Note that in General Relativity such curves are called \emph{achronal}. For strongly causal continuous Lorentzian metrics, this notion of causality
coincides with the usual one (\cite[Prop.\ 5.9]{KS:18}). In analogy to the theory
of metric length spaces, the length of a causal curve is defined via the time separation
function:  For $\gamma \colon [a,b]\rightarrow Y$ future-directed causal we set
\[
L_\tau(\gamma):=
\inf\Big\{\sum_{i=0}^{N-1} \tau(\gamma(t_i),\gamma(t_{i+1})): a=t_0<t_1<\ldots<t_N=b,\ N\in\N\Big\}.
\ 
\]
For smooth and strongly causal spacetimes $(M,g)$ this notion of length coincides with the usual one:
$L_\tau(\gamma)=L_g(\gamma)$ (\cite[Prop.\ 2.32]{KS:18}).
A future-directed causal curve $\gamma \colon [a,b]\rightarrow Y$ is 
\emph{maximal\/} if it realizes the time separation, i.e., if $L_\tau(\gamma) = \tau(\gamma(a),\gamma(b))$. Standard 
causality conditions can also be imposed on Lorentzian pre-length spaces, and substantial parts of the causal ladder 
(\cite{MS:08, Min:19b}) continue to hold in this general setting, cf.\ \cite[Subsec.\ 3.5]{KS:18}.

Lorentzian length spaces are close analogues of metric length spaces in the sense that the time separation function
can be calculated from the length of causal curves connecting causally related points. A \LpLS that satisfies some 
additional technical assumptions (cf.\ \cite[Def. 3.22]{KS:18}) is called a \emph{\LLSn\/} if
$\tau = \mathcal{T}$, where for any  $x,y\in Y$ we set
\begin{equation*}
\mathcal{T}(x,y):= \sup\{L_\tau(\gamma):\gamma \text{ future-directed causal from }x \text{ to } y\}\,, 
\end{equation*}
if the set of future-directed causal curves from $x$ to $y$ is not empty. Otherwise let $\mathcal{T}(x,y):=0$.

Any smooth strongly causal spacetime is an example of a Lorentz\-ian length space. More generally,
spacetimes with low regularity metrics and certain Lorentz-Finsler spaces \cite{Min:19a} provide further examples, cf.\ \cite[Sec.\ 5]{KS:18}.

Finally, by a \emph{timelike geodesic triangle\/} we mean a triple $(x,y,z)\in Y^3$ 
with $x\ll y \ll z$ such that $\tau(x,z)<\infty$ and such that the sides are realized by future-directed causal curves 
(that is, there exist future directed causal curves $\alpha, \beta,\gamma$ from $x$ to $y$, from $y$ to $z$ and from $x$ to $z$, respectively, with $L_\tau(\alpha)=\tau(x,y)$, $L_\tau(\beta)=\tau(y,z)$ and $L_\tau(\gamma)=\tau(x,z)$). 
Curvature bounds are formulated by comparing such triangles with triangles of the same side lengths in one of the Lorentzian 
model spaces $\mathbb{L}^2(K)$ of constant sectional curvature $K$. Here, 
\begin{equation}\label{eq:model_spaces}
\mathbb{L}^2(K) = \left\{ \begin{array}{ll}
\tilde S^2_1(r) & K=\frac{1}{r^2}\\
\R^2_1 & K=0\\
\tilde H^2_1(r) & K= -\frac{1}{r^2}\,,
\end{array}
\right.
\end{equation}
where  $\tilde S^2_1(r)$ is the simply connected covering manifold of the two-dimensional Lorentzian pseudosphere
$S^2_1(r)$ (de Sitter space), $\R^2_1$ is two-dimensional Minkowski space, and $\tilde H^2_1(r)$ is the simply connected covering manifold
of the two-dimensional Lorentzian pseudohyperbolic space (anti-de Sitter space). In order to guarantee the existence of comparison 
triangles in one of the model spaces, one needs to impose size restrictions on the sides, see \cite[Lem.\ 4.6]{KS:18}.

Using this terminology, a Lorentzian pre-length space \Yll is said to have timelike curvature bounded below (above) by 
$K\in\R$ if every point in $Y$ has a neighborhood $U$ such that:
\begin{enumerate}[label=(\roman*)] 
	\item $\tau|_{U\times U}$ is finite and continuous.
	\item Whenever $x,y \in U$ with $x \ll y$, there exists a causal curve $\alpha$ in $U$ with $L_\tau(\alpha) = 
\tau(x,y)$.
	\item If $(x,y,z)$ is a timelike geodesic triangle in $U$, realized by maximal causal curves $\alpha, 
\beta, \gamma$ whose side lengths satisfy the appropriate size restrictions, and if $(x',y',z')$ is a 
comparison triangle of 	$(x,y,z)$ in $\mathbb{L}^2(K)$ realized by timelike geodesics $\alpha '$, 
$\beta '$, $\gamma '$, then whenever $p$, $q$ are points on the sides of $(x,y,z)$ and $p', q'$ 
are corresponding points\footnote{This means that $p'$ lies on the side corresponding to the side containing $p$ at the same time separation of the vertex (i.e., e.g.\ if $p$ lies on the side $xy$ then $\tau(x,p)=\tau'(x',p')$, etc.). Similarly for $q'$.} of $(x',y',z')$, we have $\tau(p,q)\le \tau '(p',q')$ 
$($respectively $\tau(p,q)\ge \tau '(p',q'))$.
\end{enumerate}
We call such a $U$ a comparison neighborhood.

\section{Minkowski cones over metric spaces}\label{sec-con}
As a first explicit example we consider cones over metric spaces. 
Such spaces are very well-behaved and allow direct calculations even of spacelike distances. However, here we consider cones exclusively as \LLSn s, providing more details than in \cite{Ale:19}, where such cones are considered in the setting of \emph{Lorentzian pseudometric spaces}. In particular, they furnish instances of \emph{generalized cones} as defined in Section \ref{sec:gen_cone} (cf.\ Example \ref{ex:conesaregencones}).
\medskip

Proceeding by analogy with the metric geometry notion of cones over metric spaces (cf.\ \cite[Subsec.\ 3.6.2]{BBI:01}) we introduce the following:
For $X$ a geodesic length space, the \emph{Minkowski cone} $Y=\mathrm{Cone}(X)$ is defined as the quotient of $[0,\infty)\times X$ resulting from identifying all points of the form $(0,p)$. We equip $Y$ with the cone metric $d_c$ as in \cite[Def.\ 3.6.16]{BBI:01} (however this choice is not important, as it suffices to pick some background metric on $Y$ that induces the quotient topology on $[0,\infty)\times X$, to turn it into a Lorentzian pre-length space, see below). The equivalence class of $\{0\}\times X$ in $Y$ is called the \emph{vertex} of $Y$ and is denoted by $0_Y$.

\begin{rem}\label{model}
As a preparation for the following definition of the time separation function, consider $n$-dimensional Minkowski-space $\R^n_1$,
with scalar product $\langle x,y \rangle = -x_0 y_0 + \sum_{i=1}^{n-1}x_i y_i$. Then $(n-1)$-dimensional hyperbolic space $\hyp^{n-1}$ is isometrically embedded into $\R^n_1$ as $\{x\in\R^n_1\mid \langle x,x\rangle = -1,\ x_0>0  \}=\{x\in\R^n_1 \mid \tau_{\R^n_1}(0,x)=1,\  x_0>0  \}=:\Sigma $, where $\tau_{\R^n_1}$ is the time separation function on $\R^n_1$. Let us denote this embedding by $\psi$. The induced Riemannian distance function on $\hyp^{n-1}$ is uniquely determined by $\cosh d_{\hyp^{n-1}} (x,y) = -\langle \psi(x),\psi(y)\rangle$ for $x,y\in \hyp^{n-1}$.
Suppose now that $x,y\in \hyp^{n-1}$ and let $s,t>0$. Then $\psi(x)=(\sqrt{1+|x'|^2},x')$ for some $x'\in \R^{n-1}$, and analogously
for $y$. Setting $\theta:=d_{\hyp^{n-1}} (x,y)$ we calculate
\begin{align*}
\langle t\psi(y)-s\psi(x),t\psi(y)-s\psi(x)\rangle &=-t^2 -s^2 -2st \langle \psi(x),\psi(y)\rangle  \\ &= -(s^2+t^2 - 2st \cosh \theta).
\end{align*}
This shows that the quotient $Y$ of $[0,\infty)\times \hyp^{n-1}$ modulo $(0,x)\sim (0,y)$ for all $x,y$
can be identified with the cone $I^+(0)\cup\{0\}\subseteq \R^n_1$ via $(s,x)\mapsto s\psi(x)$, and restricting this identification, we see that $(0,\infty)\times \hyp^{n-1}$ corresponds to $I^+(0)\subseteq \R^n_1$. Pulling back the causal structure and time separation of $\R^n_1$ we make the following definitions: Two points $(s,x)$ and $(t,y)$ in $Y$ are said to satisfy $(s,x) \leq_Y (t,y)$ if and only if $s\le t$ and $s^2+t^2 - 2st \cosh \theta \ge 0$, which is equivalent to $sx \le ty$ in $\R^n_1$. In addition, for $(s,x) \le_Y (t,y)$ the Minkowski cone time separation function $\tau_Y $ is defined by
$\tau_Y((s,x),(t,y))=\tau_{\R^n_1}(s\psi(x),t\psi(y))= \sqrt{s^2+t^2 - 2st \cosh \theta}$, and otherwise $\tau_Y((s,x),(t,y))=0$.

For future reference, let us briefly remark that the Minkowski cone time separation $\tau_Y$ defined above induces a time separation $\tau_C$ on $C:=(0,\infty)\times \hyp^{n-1}$ via restriction and this restriction equals the time separation $\tau_W$ of the Lorentzian warped product manifold $W:=(0,\infty)\times_\mathrm{id} \hyp^{n-1}=\bigl((0,\infty)\times \hyp^{n-1},g:=-dt^2+t^2\langle .,. \rangle_{\hyp^{n-1}}\bigr)$: 

It suffices to show that the map $\varphi : W \to I^+(0)\subseteq \R^n_1$, $(s,x)\mapsto s\psi(x)$ is an isometry,
as this will imply $\tau_W((s,x),(t,y))=\tau_{\R^n_1}(s\psi(x),t\psi(y))=\tau_C((s,x),(t,y))$. We have 
$D\varphi|_{(r,z)}=r \, D\psi|_{z}\circ \mathrm{pr}_{T\hyp^{n-1}}+\psi(z) \mathrm{pr}_{T\R_+}$,
hence for vectors $(S,X),(T,Y)\in T_r\R_+ \times T_z\hyp^{n-1}$ we get
\begin{align*}g((S,X),(T,Y))=&-ST+r^2 \langle X,Y\rangle_{\hyp^{n-1}}=-ST+r^2 \langle D\psi X, D\psi Y\rangle_{\R^n_1}\\
=& \langle S\psi(z),T\psi(z) \rangle_{\R^n_1} +\langle r D\psi|_z X, r D\psi|_z Y\rangle_{\R^n_1} \\
=& \langle D\varphi (S,X), D\varphi (T,Y)\rangle_{\R^n_1},
\end{align*}
because $\langle \psi(z),D\psi|_z X \rangle_{\R^n_1}=0$ for any $X\in T_z\hyp^{n-1}$.
So $\varphi $ is an isometry, as claimed.

\end{rem} 
Coming back to the general case, to equip the cone $Y$ that results from identifying all points with first component $0$ in
$[0,\infty)\times X$ as defined above with a time separation function, we proceed analogously, with the metric
$d_X$ of $X$ taking over the role of $\theta=d_\mathbb{H}$ from Remark \ref{model}. Thus we say that $(s,p)\le (t,q)$
(resp.\ $(s,p)\ll (t,q)$) if $s\le t$ and $s^2+t^2 - 2st \cosh d_X(p,q) \ge 0$ (resp.\ $>0$ in both cases).
Then 
\begin{equation}\label{taudef}
\tau((s,p),(t,q)):= \sqrt{s^2+t^2 - 2st \cosh d_X(p,q)},
\end{equation}
and $\tau((s,p),(t,q)) := 0$ otherwise.

\begin{prop}\label{Minklpls} $(Y,d_c,\ll,\le,\tau)$ is a Lorentzian pre-length space. Moreover, $\tau$ is continuous.
\end{prop}
\begin{proof}
Since the causal and timelike relations are defined via $\tau$, and since $\tau$ is clearly continuous with respect
to $d_c$, it only remains to check the reverse triangle inequality for $\tau$. So let
$(s,p), (t,q), (u,r) \in Y$ and fix three comparison points $\tilde p,\tilde q,\tilde r$ in $\hyp^2 \subseteq \R^3_1$ with $d_{\hyp^2}(\tilde p,\tilde q) = d_X(p,q)$, and so on (note that some of these distances might be zero). Then from Remark
\ref{model} and the definition of $\tau$ it follows that (denoting by $\tilde \tau$ the time separation function in 
$\R^3_1$), 
$\tilde \tau(s\tilde p, t\tilde q) = \tau((s,p),(t,q))$, etc. The reverse triangle inequality for $\tau$ therefore is immediate from that of $\tilde \tau$.
\end{proof}

\begin{lem}\label{curve_rel} Suppose that $0_Y \neq (s,p)\ll (t,q)\in Y$.
\begin{itemize}
\item[(i)]  Let $\gamma: [0,a] \to Y$, $\gamma(\lambda)=(r(\lambda),\sigma(\lambda))$ be a maximizing timelike curve from 
$(s,p)$ to $(t,q)$. Then $\sigma$ is a minimizing geodesic from $p$ to $q$ in $X$.
\item[(ii)] Conversely, suppose that $\sigma$ is a minimizing geodesic from $p$ to $q$ in $X$.
Let $\tilde y_0$ and $\tilde y_1$ be points in $I^+(0)\subseteq\R^2_1$ with distance $r(0):=s$ resp.\ $r(a):=t$ from $0$ and 
such that the hyperbolic angle $\mathrm{arcosh}(-\frac{\langle \tilde y_0,\tilde y_1\rangle}{s t})$ between them is $d_X(p,q)$.
For $\lambda\in [0,a]$, let $r(\lambda)$ be the distance of the intersection of the straight line
connecting $\tilde y_0$ to $\tilde y_1$ with the half-ray in $I^+(0)$ that has hyperbolic
angle $d_X(p,\sigma(\lambda))=\lambda$ with the half-ray through $\tilde y_0$. Then $\lambda\mapsto 
(r(\lambda),\sigma(\lambda))$ is a $\tau$-realizing curve from $(s,p)$ to $(t,q)$ in $Y$.
\end{itemize}
\end{lem}
\begin{proof}
(i) Let $\lambda_1<\lambda_2<\lambda_3 \in [0,a]$, let $y_i:=(r(\lambda_i),\sigma(\lambda_i))$ 
and pick points $\tilde y_i$ ($i=1,2,3$) in $I^+(0)\subseteq\R^2_1$ such that their distance from $0$ is $r(\lambda_i)$, 
the hyperbolic angle  between $\tilde y_1$ and $\tilde y_2$ is
$d_X(\sigma(\lambda_1),\sigma(\lambda_2))$, and the hyperbolic angle  between $\tilde y_2$ and $\tilde y_3$ is
$d_X(\sigma(\lambda_2),\sigma(\lambda_3))$. This means that $\tau(y_1,y_2)=
\tilde\tau(\tilde y_1,\tilde y_2)$, as well as $\tau(y_2,y_3)=
\tilde\tau(\tilde y_2,\tilde y_3)$. Now by assumption, $\tau(y_1,y_3)=\tau(y_1,y_2)+\tau(y_2,y_3)$,
and the ensuing equality for $\tilde \tau$ implies that the $\tilde y_i$ 
must lie on a straight line in $\R^2_1$. Consequently,
their hyperbolic angles must add up, i.e., $d_X(\sigma(\lambda_1),\sigma(\lambda_2))+
d_X(\sigma(\lambda_2),\sigma(\lambda_3))=d_X(\sigma(\lambda_1),\sigma(\lambda_3))$. It 
follows that $\sigma$ is indeed distance-realizing.

\noindent(ii) This is straightforward from the definition of $\tau$ and Remark \ref{model}.
\end{proof}
As an immediate consequence of Lemma \ref{curve_rel} (ii) (and the obvious fact that $s\mapsto (s,q)$ is a realizing geodesic from $0_Y\equiv (0,q)$ to $(t,q)$ for all $t>0,q\in X$) we obtain:
\begin{cor}
Any two causally related points in $Y$ can be connected by a realizing geodesic, i.e., $Y$ is \emph{geodesic}.
\end{cor}

The following result establishes, in the present setting, a relation between metric curvature bounds in the Alexandrov space $X$ and timelike curvature bounds in the Minkowski cone $Y$ over $X$, foreshadowing analogous results for generalized cones in Section \ref{sec-cb}. In particular, the following theorem is a special case of Theorems \ref{thm-Y-cb-X-cb} and \ref{thm-X-cb-Y-cb} below (and is analogous to the result in the metric case, cf.\ \cite[Thm.\ 4.7.1]{BBI:01}).

\begin{thm}
Let $Y=\mathrm{Cone}(X)$ be the Minkowski cone over a geodesic length space $X$. Then $Y$ has timelike curvature
bounded below (above) by $0$ if and only if $X$ is an Alexandrov space of curvature bounded below (above) by $-1$.
\end{thm}
\begin{proof}
We observe that timelike comparison triangles for $Y$ and comparison triangles for $X$ can be related in
the following way: Let $(s,p) \ll (t,q) \ll (u,r) \in Y$ be the vertices of a timelike triangle $\Delta$ in $Y$.  If $(s,p)\neq 0_Y$, choose three comparison points $\tilde p,\tilde q,\tilde r$ in 
$\hyp^2 \subseteq \R^3_1$ with $d_{\hyp^2}(\tilde p,\tilde q) = d_X(p,q)$, and so on (note that $\tilde p,\tilde q,\tilde r$ need not be pairwise distinct). If $(s,p) = 0_Y$, choose two points $\tilde{q},\tilde{r}\in \hyp^2$ with $d_{\hyp^2}(\tilde p,\tilde q) = d_X(p,q)$. Then by definition
of $\tau$, the points $s\cdot \tilde p,\ t\cdot \tilde q$ and $u\cdot \tilde r$ in $\R^3_1$ form a timelike comparison triangle $\tilde 
\Delta$ for $((s,p), (t,q), (u,r))$ (note that $s\cdot \tilde{p}=0$ if $(s,p)=0_Y$ and that the points $s\cdot \tilde p,\ t\cdot \tilde q,\ u\cdot\tilde{r}$ will always be pairwise distinct if $(s,p), (t,q), (u,r)$ are, even if $\tilde p,\tilde q,\tilde r$ are not).
Indeed, their $\tilde\tau$-side lengths in $\R^3_1$ are exactly the $\tau$-lengths of the original triangle in $Y$, and so equivalently we may use the two-dimensional Minkowski space $M$ spanned by $s\cdot \tilde p, t\cdot \tilde q, u\cdot \tilde r$ as a (flat) comparison space for $\Delta$: 
$M$ is clearly totally geodesic in $\R^3_1$, so its time separation function is precisely the restriction of $\tilde \tau$ to 
$M\times M$. Thus $\tilde\Delta$ can just as well be viewed as
a subset of $M$. 

Suppose now, first, that $Y$ has timelike curvature bounded below by $0$. Let $p_0\in X$ and let $V\subseteq  Y$ be a neighborhood of $(1,p_0)$ on which timelike comparison holds. Then there exists $\eps >0$ and a neighborhood $U\subseteq X$ of $p_0$ such that $(1-\eps,1+\eps)\times U\subseteq V$ and any triangle $(p,q,r)$ in $U$ can be lifted to a timelike triangle $(s,p)\ll (t,q) \ll (u,r)$ in $V$ (the last requirement follows from \eqref{taudef} and Lemma \ref{curve_rel} (ii) by shrinking $U$ but keeping $\eps $ fixed). Let $(p,q,r)$ form a triangle in $U$. Also, let $m,n$ be points on the sides of $(p,q,r)$ and $\tilde m, \tilde n$ be corresponding points on the 
	sides of a comparison triangle $(\tilde p, \tilde q,\tilde{r})$ in $\hyp^2$. W.l.o.g.~(renaming points if necessary) $m$ lies on the side from $p$ to $q$ and $n$ on the side from $q$ to $r$.
 Given realizing geodesics in $U$ for the edges of $(p,q,r)$, by Lemma \ref{curve_rel} (ii) we obtain
corresponding realizing geodesics for the sides of the triangle $\Delta=((s,p),(t,q),(u,r))$
in $V$. Note that the points $M=(r_{pq}(\lambda_m), m)$ and $N=(r_{qr}(\lambda_n),n)$ on these geodesics satisfy $(s,p)\leq M \leq (t,q) \leq N \leq (u,r) $  and are timelike related (or equal). Let $\tilde M = \ell_{\tilde{m}} \cdot \tilde{m}$, $\tilde N = \ell_{\tilde{n}}  \cdot \tilde{n}$ be points in $\R^3_1$ on the sides of the comparison triangle $\tilde \Delta:=(s\cdot \tilde p,\ t\cdot \tilde q,\ u\cdot\tilde{r})$ in $\R^3_1$. From the construction of $r(\lambda) $ in Lemma \ref{curve_rel} (ii) we see that $r_{pq}(\lambda_m)=\ell_{\tilde{m}}$ and $r_{qr}(\lambda_n)=\ell_{\tilde{n}}$. So $M, N \in \Delta$ and $\tilde M , \tilde N \in \tilde \Delta$ are corresponding points and by
 \eqref{taudef} and the monotonicity of $\cosh$ it then follows that 
$d_X(m,n)\ge d_{\hyp^2}(\tilde m,\tilde n)$, because $\tau(M,N)\le \tilde\tau(\tilde M,\tilde N)$.

Conversely, if $X$ has curvature bounded below by $-1$, then a similar (in fact, easier) argument, this time based on Lemma 
\ref{curve_rel} (i), shows that $Y$ has timelike curvature bounded below
by $0$.

Bounds from above can be treated analogously.
\end{proof}

\section{Generalized cones}\label{sec:gen_cone}
In this section, we introduce a generalization of warped products of metric spaces to the 
Lorentzian setting.

\begin{defi}
For $(X,d)$ a metric space and $I\subseteq\R$ an open interval, set $Y:=I\times X$ and put the product metric on 
$Y$, i.e., $D((t,x),(t',x')) = |t-t'| + d(x,x')$ for $(t,x),(t',x')\in Y$. Let $f\colon I\rightarrow (0,\infty)$ be 
continuous. Then $Y\equiv I\times_fX$ is called a \emph{generalized cone} and $f$ is called \emph{warping 
function}. With this notation, i.e., $I\times_f X$, we indicate that the Lorentzian structure (to be introduced below) on the product $I\times X$ can be thought of as ``$-dt^2 + f^2 d_X^2$''.
\end{defi}
Alternatively, generalized cones can also be called \emph{(Lorentzian) warped products 
with one-dimensional base}. Henceforth, all topological notions refer to the metric topology induced by $D$. Note,
however that the concrete form of the metric on $I\times X$ plays no role as long as it induces the given
metric structures on $I$ and $X$, respectively.

We first turn to the question of introducing an appropriate Lorentzian structure on a \wpd. 
To this end, we have to define causal curves.
\begin{defi}\label{def-cc}
 Let $Y=I\times_f X$ be a \wpd\ and let $\gamma\colon J\rightarrow Y$ be an absolutely continuous curve (with respect to 
$D$). Such a curve has components $\gamma=(\alpha,\beta)$, where $\alpha\colon J\rightarrow I$ and $\beta\colon 
J\rightarrow X$ are both absolutely continuous, and the metric derivative of $\beta$, $v_\beta$, exists 
almost everywhere (cf.\ \cite[Thm.\ 1.1.2]{AGS:05}). We additionally require that $\alpha$ is strictly monotonous. The curve $\gamma$ is called
\begin{equation}
\begin{cases}
 \text{timelike}\\
 \text{null}\\
 \text{causal}
 \end{cases}
 \text{\quad if \qquad}
 -\dot\alpha^2 + (f\circ \alpha)^2 v_{\beta}^2\quad
 \begin{cases}
  < 0\\
  = 0\\
  \leq 0\,,
 \end{cases}
\end{equation}
almost everywhere. It is called \emph{future/past directed causal} if $\alpha$ is strictly monotonically 
increasing/decreasing, i.e., $\dot\alpha>0$ or $\dot\alpha<0$ almost everywhere.
\end{defi}

\begin{rem}
 So far in the development of the theory of \LLSn s, locally Lipschitz continuous curves were used as causal curves. 
However, as we shall establish in Lemma \ref{lem-ac-lip} below, every absolutely continuous causal curve has a parametrization 
as a Lipschitz curve. So using absolutely continuous curves is compatible with the previous works \cite{KS:18, GKS:19}. 
Moreover, parametrizing a timelike curve with respect to arclength anyways only gives an absolutely continuous curve in general 
--- an issue that also necessitated a special treatment in \cite[Subsec.\ 3.7]{KS:18}. Analogous questions arise for spacetimes with continuous metrics, in which case we refer to \cite{GKSS:20}.
\end{rem}

\begin{defi}(Length of a causal curve)\label{def-len}
Let $\gamma=(\alpha,\beta)\colon [a,b]\rightarrow Y$ be a causal curve. Its \emph{length} $L(\gamma)$ is defined as 
\begin{equation}
 L(\gamma):= \int_a^b \sqrt{\dot\alpha^2 - (f\circ\alpha)^2 v_\beta^2}\,.
\end{equation}
\end{defi}

\begin{rem}\label{rem-len-int-ac}
 Note that $\sqrt{\dot\alpha^2 - (f\circ\alpha)^2 v_\beta^2}$ is integrable as $\alpha$ is absolutely continuous, 
$f$ is bounded on the compact image of $\alpha$ and the metric derivative is integrable by \cite[Thm.\ 1.1.2]{AGS:05}. 
Moreover, from this it follows that the map $t\mapsto L(\gamma\rvert_{[a,t]}) = \int_a^t \sqrt{\dot\alpha^2 - 
(f\circ\alpha)^2 v_\beta^2}$ is absolutely continuous.
\end{rem}

\begin{lem}\label{lem:reparametrization}
 Let $(Z,\rho)$ be a metric space, $J, J'$ intervals, $\lambda\colon J\rightarrow Z$ an absolutely continuous curve and 
$\phi\colon J'\rightarrow J$ strictly monotonous and such that 
both $\phi$ and $\phi^{-1}$ are absolutely continuous. 
Then $\xi:=\lambda\circ\phi$ is absolutely continuous and
\begin{equation}
 v_\xi = (v_\lambda\circ\phi) |\phi'|\,.
\end{equation}
\end{lem}
\begin{pr}
 That $\xi$ is absolutely continuous follows as in \cite[Thm.\ 3, Ch.\ IX, \textsection 1]{Nat:55}. Moreover, the metric 
derivative of $\lambda$ exists almost everywhere (cf.\ \cite[Thm.\ 1.1.2]{AGS:05}), so let $t\in J$ be such a point. Then 
for any $h\in\R$ such that $t+h\in J$ we have 
\begin{equation}\label{eq-met-der-rep}
 \rho(\lambda(t+h),\lambda(t)) = |h|v_\lambda(t) + r(h)\,,
\end{equation}
where the remainder term satisfies $r(h)=o(|h|)$. Next, let $s\in J'$ such that both $\phi'(s)$ and $\lambda'(\phi(s))$ exist. 
The set of all such $s$ has full measure in $J'$ because the absolutely continuous function $\phi^{-1}$ maps sets of measure zero to 
sets of measure zero (Lusin's property, cf.\ e.g.\ \cite[Thm.\ 3.4.3]{AT:04}). 
Furthermore, let $h\in\R$ with $s+h\in J'$. We conclude that
\begin{align}
 &\rho(\xi(s+h),\xi(s))\\
 &=\rho(\lambda(\phi(s+h)),\lambda(\phi(s))) \mathop{=}^{\eqref{eq-met-der-rep}} |h\phi'(s) + 
r(h)|v_\lambda(\phi(s)) + r(h\phi'(s)+r(h))\\
&=|h||\phi'(s)|v_\lambda(\phi(s)) + \underbrace{(|h\phi'(s) + 
r(h)|-|h||\phi'(s)|)}_{=o(|h|)}v_\lambda(\phi(s)) + o(|h|) \\
&= |h||\phi'(s)|v_\lambda(\phi(s)) + o(|h|)\,,
\end{align}
yielding the claim.
\end{pr}

A direct corollary of the above lemma is that the length of causal curves is invariant under reparametrizations.
\begin{cor}\label{cor:reparametrization}
 The length $L$ is reparametrization invariant, i.e., if $\gamma=(\alpha,\beta)$ is a causal curve defined on some interval 
$J$ and $\phi:J'\to J$ is strictly increasing 
and such that $\vphi$ and $\vphi^{-1}$ are absolutely continuous, 
then $\gamma\circ\phi$ is a causal 
curve of the same length and time orientation.
\end{cor}

\begin{rem}
	Note that this means that $Y$ with these future/past directed causal/timelike curves and this length functional is an example of a \emph{Lorentzian length structure} as defined in the Appendix, see Definition \ref{def-LLStr}.
\end{rem}
To establish that the length functional $L$ is upper semicontinuous (with respect to pointwise convergence) we need to 
describe the length in a variational way. As we show below, the variational length is the same as the length defined above via 
the (metric) derivative of the curve. 
\begin{defi}
 Let $\gamma=(\alpha,\beta)\colon[a,b]\rightarrow Y$ be a causal curve. For $s,t\in I$, $s \leq t$,  
set $m_{s,t}:=\min_{r\in[s,t]} f(r)>0$. Then the \emph{variational length} of $\gamma$ is defined as $\lv(\gamma):=$
\begin{equation}
 \inf_{a=t_0<t_1<\ldots<t_N=b}\sum_{i=0}^{N-1} \sqrt{(\alpha(t_{i+1})-\alpha(t_i))^2 - 
m_{\alpha(t_i),\alpha(t_{i+1})}^2 d(\beta(t_i),\beta(t_{i+1}))^2}\,.
\end{equation}
\end{defi}
To see that $\lv$ is well-defined we need the following lemma.
\begin{lem}\label{lem-der-var-geq-0}
Let $(X,d)$ be a metric space and let $\gamma=(\alpha,\beta)\colon[a,b]\rightarrow Y$ be a causal curve. Then for
any $a\leq s \leq t \leq b$ we have:
 \begin{equation}
(\alpha(t)-\alpha(s))^2 - m_{\alpha(s),\alpha(t)}^2 d(\beta(s),\beta(t))^2 \geq 0\,.
 \end{equation}
\end{lem}
\begin{pr}
 Without loss of generality let $\gamma$ be future directed, i.e., $\dot\alpha>0$ almost everywhere. Since $\gamma$ is causal 
we have $(f\circ \alpha)^2 v_\beta^2\leq \dot\alpha^2$ and, as all involved quantities are non-negative, in fact $(f\circ 
\alpha) v_\beta\leq \dot\alpha$. Denote by $L^d$ the length functional of $(X,d)$, then
\begin{align}
 m_{\alpha(s),\alpha(t)}\, L^d(\beta\rvert_{[s,t]}) = m_{\alpha(s),\alpha(t)} \int_s^t v_\beta \leq \int_s^t (f\circ\alpha) 
v_\beta \leq \int_s^t\dot\alpha = \alpha(t)-\alpha(s)\,.
\end{align}
Finally, as $d(\beta(s),\beta(t))\leq L^d(\beta\rvert_{[s,t]})$ we conclude that $m_{\alpha(s),\alpha(t)}^2 
d(\beta(s),\beta(t))^2$ $\leq (\alpha(t)-\alpha(s))^2$.
\end{pr}
Also note that the variational length is invariant under reparametrizations as it is defined via partitions, cf.\ e.g.\ the 
proof of \cite[Prop.\ 1.1.8]{Pap:14}. Additionally, $\lv$ is additive which is easily inferred from the inequality (ii) in the next Lemma.
\medskip

\begin{lem}\label{lem-rev-tri-ine}
 Let $a,b\in I$ with $a\leq s\leq t\leq u\leq b$ and $x,y,z\in X$ such that
 \begin{align}
  \label{eq-rev-tri-ine-st} &(t-s)^2 - m_{s,t}^2 d(x,y)^2\geq 0\,,\\
  \label{eq-rev-tri-ine-tu} &(u-t)^2 - m_{t,u}^2 d(y,z)^2\geq 0\,.
 \end{align}
Then
\begin{enumerate}
 \item[(i)] $(u-s)^2 - m_{s,u}^2 d(x,z)^2\geq 0$, and
 \item[(ii)] $\sqrt{(t-s)^2 - m_{s,t}^2 d(x,y)^2} + \sqrt{(u-t)^2 - m_{t,u}^2 d(y,z)^2}$\\
 $\leq \sqrt{(u-s)^2 - m_{s,u}^2 d(x,z)^2}$.
\end{enumerate}
\end{lem}
\begin{pr}
 Clearly, $m_{s,u}= \min(m_{s,t},m_{t,u})$ and without loss of generality we may assume that $m_{s,u}=m_{s,t}$. Let $(X,Y,Z)$ be a 
comparison triangle of $(x,y,z)$ in the plane $\R^2$, i.e., $\|X-Y\| = d(x,y)$, $\|X-Z\|=d(x,z)$ and $\|Y-Z\|=d(y,z)$. For 
$c>0$ define the scaled Minkowski metric $\eta_c$ on $\R^3$ as $\eta_c:= -(dx^0)^2 + c ((dx^1)^2 + (dx^2)^2)$. We claim that 
$(s,X)\leq (t,Y)\leq (u,Z)$ in $(\R^3,\eta_{m_{s,t}^2})$. That $(s,X)\leq (t,Y)$ follows directly from 
\eqref{eq-rev-tri-ine-st}, and that $(t,Y)\leq (u,Z)$ follows from \eqref{eq-rev-tri-ine-tu} and $m_{s,t}\leq m_{t,u}$. Thus 
$(t,X)\leq (u,Z)$ by the transitivity of the causal relation $\leq$ in $(\R^3,\eta_{m_{s,t}^2})$, giving (i). 

To show (ii), denote by $P$ the time 
separation function of $(\R^3,\eta_{m_{s,t}^2})$. Then  
\begin{align}
 P((s,X),(t,Y)) &= \sqrt{(t-s)^2 - m_{s,t}^2 d(x,y)^2},\,\text{ and}\\
 P((t,Y),(u,Z)) &= \sqrt{(u-t)^2 - m_{s,t}^2 d(y,z)^2} \geq \sqrt{(u-t)^2 - m_{t,u}^2 d(y,z)^2}\,.
\end{align}
Consequently, by the reverse triangle inequality for $P$ we obtain
\begin{align}
 &\sqrt{(t-s)^2 - m_{s,t}^2 d(x,y)^2} + \sqrt{(u-t)^2 - m_{t,u}^2 d(y,z)^2}\\
 &\qquad\leq P((s,X),(t,Y)) + P((t,Y), (u,Z)) \leq P((s,X),(u,Z))\\
 &\qquad= \sqrt{(u-s)^2 - m_{s,u}^2 d(x,z)^2}\,,
\end{align}
as $m_{s,u} = m_{s,t}$.
\end{pr}

\begin{rem}
 The above Lemma \ref{lem-rev-tri-ine} shows that the function \[T((t,x),(s,y)):=(t-s)^2 - m_{s,t}^2 d(x,y)^2,\] if 
non-negative, and $T((t,x),(s,y)):=0$ otherwise, satisfies the reverse triangle inequality. So in principle it could also be
used to define a Lorentzian (pre-)length space. However, as it only involves the minimum of $f$ on the interval $[s,t]$ it 
does not contain the full information of $f$ on this interval and it is not compatible with the smooth case (i.e., if $X$ is 
a smooth Riemannian manifold and $f$ is smooth). Despite this, it proves very useful when handling the variational length because the Lorentzian (pre-)length space definition of length in $(I\times X,\ll,\leq,T)$ coincides with the variational length $\lv$ defined above. We will show in Proposition \ref{prop-l-lv} that the variational length equals the length defined in Definition \ref{def-len}.
\end{rem}
\begin{lem}\label{lem-ac-lip}
Every future directed causal curve has a reparametrization that is locally Lipschitz 
continuous. In particular, any future directed causal curve defined on a compact interval $\gamma\colon[a,b]\rightarrow Y$ 
has a reparametrization $\tilde\gamma$ such that $\tilde\gamma(t)=(t,\tilde\beta(t))$. Similarly, $\gamma $ has a reparametrization $\gamma'=(\alpha',\beta')$ such that $\beta':[0,L^d(\beta')]\to X$ is parametrized with respect to arc length.
\end{lem}
\begin{pr}
 As this is a local question we may restrict to the case of compact intervals. 
Let $\gamma=(\alpha,\beta): [a,b]\rightarrow Y$ be future 
directed causal, 
then $\dot\alpha>0$ almost everywhere.  Thus by a theorem of Zareckii (cf.\ \cite[p.\ 271]{Nat:55}), 
$\alpha^{-1}$ is absolutely continuous, hence can serve as an admissible parametrization for $\gamma$. 
Then $\tilde\gamma:=\gamma\circ\alpha^{-1}$ satisfies 
$\tilde\gamma(t)=(t,\tilde\beta(t))$, where 
$\tilde\beta:=\beta\circ\alpha^{-1}$. 
By Corollary \ref{cor:reparametrization}, $\tilde\gamma$ is future directed causal and so we have 
$f^2v_{\tilde\beta}^2\leq 1$. 
Thus $v_{\tilde\beta}\leq \frac{1}{C}$, where $C:=\min_{r\in[\alpha(a),\alpha(b)]}f(r)>0$. This implies that $\tilde\beta$ is 
Lipschitz 
continuous, due to
\begin{align}
 d(\tilde\beta(s),\tilde\beta(t))\leq L^d(\tilde\beta\rvert_{[s,t]}) = \int_s^t v_{\tilde\beta}\leq \frac{1}{C} |t-s|\,,
\end{align}
where $\alpha(a)\leq s < t\leq \alpha(b)$.
\end{pr}

\begin{prop}\label{prop-l-lv}
Let $(X,d)$ be a metric space. Then the variational length of any causal curve $\gamma$ in $Y=I\times_fX$ 
agrees with its length, i.e., $L(\gamma)=\lv(\gamma)$.
\end{prop}
\begin{pr}
 Let $\gamma$ be a (without loss of generality) future directed causal curve. As $\lv$ and $L$ are invariant under 
reparametrizations, using Lemma \ref{lem-ac-lip} we may assume that $\gamma\colon[a,b]\rightarrow Y$ is parametrized as 
$\gamma(t)=(t,\beta(t))$ (where $a,b\in I$). Let $a\leq s< t\leq b$, then $d(\beta(s),\beta(t))\leq 
L^d(\beta\rvert_{[s,t]}) = \int_s^t v_\beta$ and so
\begin{align}
 \frac{1}{(t-s)^2}\, d(\beta(s),\beta(t))^2 \leq \frac{1}{(t-s)^2}\, (\int_s^t v_\beta)^2\leq \frac{1}{t-s} \int_s^t 
v_\beta^2\,,
\end{align}
where in the last step we used Jensen's inequality. Thus we obtain
\begin{equation}\label{eq-d-vb}
 \frac{1}{t-s}\, d(\beta(s),\beta(t))^2 \leq \int_s^t v_\beta^2\,.
\end{equation}
Again using Jensen's inequality we estimate
\begin{align}
 \Bigl(\frac{1}{t-s}\, \int_s^t \sqrt{1-f^2 v_\beta^2}\Bigr)^2 &\leq \frac{1}{t-s}\, \int_s^t (1-f^2 v_\beta^2) = 1 - 
\int_s^t f^2 \frac{v_\beta^2}{t-s}\\
&\leq 1-m_{s,t}^2\int_s^t \frac{v_\beta^2}{t-s}\mathop{\leq}^{\eqref{eq-d-vb}} 1 - m_{s,t}^2 
\frac{d(\beta(s),\beta(t))^2}{(t-s)^2}\,,
\end{align}
which yields 
\begin{equation}\label{eq-int-leq-sqr}
\begin{split}
 \int_s^t \sqrt{1-f^2 v_\beta^2} 
 &\leq (t-s)\sqrt{1 - m_{s,t}^2 \frac{d(\beta(s),\beta(t))^2}{(t-s)^2}}\\ 
 &= \sqrt{(t-s)^2 - m_{s,t}^2 d(\beta(s),\beta(t))^2}\,.
\end{split}
\end{equation}
Now let $a=t_0<t_1<\ldots<t_N=b$ be a partition of $[a,b]$, then
\begin{align}
 L(\gamma)&=\int_a^b \sqrt{1-f^2 v_\beta^2} = \sum_{i=0}^{N-1} 
\int_{t_i}^{t_{i+1}}\sqrt{1-f^2v_\beta^2}\\
&\mathop{\leq}^{\eqref{eq-int-leq-sqr}} \sum_{i=0}^{N-1}\sqrt{(t_{i+1}-t_i)^2 - 
m_{t_i,t_{i+1}}^2 d(\beta(t_i),\beta(t_{i+1}))^2}\,,
\end{align}
and taking the infimum over all partitions of $[a,b]$ gives $L(\gamma)\leq \lv(\gamma)$.
\medskip

For the reverse inequality, note that we have by definition of $\lv$ as the infimum over all partitions of the interval that
\begin{equation}\label{eq-lv-leq}
 \lv(\gamma\rvert_{[s,t]})\leq \sqrt{(t-s)^2 - m_{s,t}^2 d(\beta(s),\beta(t))^2}\,,
\end{equation}
for all $a\leq s<t\leq b$. Let $0<\eps<b-a$, set $\tilde b:=b-\eps>a$, $h:=\frac{\tilde b-a}{N}$ where $N\in\N$ is such that 
$h\leq \eps$ and set $t_i:=a + i h$ for $i=0,\ldots,N$. Then for $t\in[a,\tilde b]$ we have $t+h\in[a,b]$ and hence by Lemma 
\ref{lem-der-var-geq-0} that $h^2-m_{t,t+h}^2 d(\beta(t),\beta(t+h))^2\geq 0$ . Consequently we get
\begin{align}
 \frac{1}{h}\int_a^{\tilde b}&\sqrt{h^2-m_{t,t+h}^2 d(\beta(t),\beta(t+h))^2}\,\dd t\\
 & = \frac{1}{h}\sum_{i=0}^{N-1} \int_{t_i}^{t_{i+1}}\sqrt{h^2-m_{t,t+h}^2 d(\beta(t),\beta(t+h))^2}\,\dd t\\
 & \mathop{=}^{(*)} \frac{1}{h}\int_0^h \sum_{i=0}^{N-1} \sqrt{h^2-m_{t_i+t,t_{i+1}+t}^2 
d(\beta(t_i+t),\beta(t_{i+1}+t))^2}\,\dd t\\
&\geq \frac{1}{h}\int_0^h \lv(\gamma\rvert_{[a+t,\tilde b+t]})\,\dd t\\
& = \frac{1}{h}\int_0^h\Bigl(\lv(\gamma) - \underbrace{\lv(\gamma\rvert_{[a,a+t]})}_{\leq \lv(\gamma\rvert_{[a,a+h]})} - 
\underbrace{\lv(\gamma\rvert_{[\tilde b+t,b]})}_{\leq \lv(\gamma\rvert_{[\tilde b,b]})}\Bigr)\,\dd t\\
&\geq \lv(\gamma) - \lv(\gamma\rvert_{[a,a+h]}) - \lv(\gamma\rvert_{[\tilde b,b]})\\
&\mathop{\geq}^{\eqref{eq-lv-leq}} \lv(\gamma) 
- \sqrt{h^2-m_{a,a+h}^2 d(\beta(a),\beta(a+h))^2}\\
&\hspace*{10em} - \sqrt{\eps^2-m_{b -\eps,b}^2 d(\beta(b-\eps),\beta(b))^2}\,,
\end{align}
where we used additivity of $\lv$ (cf.~Lemma \ref{lem-rev-tri-ine}) and, in $(*)$, the substitution $t'=t+a+ih=t+t_i$. As $0\leq \frac{1}{h} \sqrt{h^2-m_{t,t+h}^2 d(\beta(t),\beta(t+h))^2}\leq 
1\in L^1([a,b])$ for all $h\geq 0$ such that $t+h\in[a,b]$, we have by dominated convergence and the above inequality that
\begin{align}
\int_a^{\tilde b} \sqrt{1-f^2 v_\beta^2} &= \lim_{h\searrow 0} \int_a^{\tilde b}\sqrt{1-m_{t,t+h}^2 
\frac{d(\beta(t),\beta(t+h))^2}{h^2}}\,\dd t\\
&\geq \lv(\gamma) - \sqrt{\eps^2-m_{b-\eps,b}^2 d(\beta(b-\eps),\beta(b))^2}\,.
\end{align}
Thus
\begin{align}
L(\gamma) &= L(\gamma\rvert_{[a,b-\eps]}) + L(\gamma\rvert_{[b-\eps,b]})\\
&\geq \lv(\gamma) - \sqrt{\eps^2-m_{b-\eps,b}^2 
d(\beta(b-\eps),\beta(b))^2} + L(\gamma\rvert_{[b-\eps,b]})\,,
\end{align}
and letting $\eps\searrow 0$ finishes the proof.
\end{pr}

\begin{prop}\label{prop-L-usc}
Let $(X,d)$ be a metric space and let $\gamma_n,\gamma$ $(n\in\N)$ be causal curves defined on the interval $[a,b]$ such 
that $\gamma_n\to\gamma$ pointwise. Then
 \begin{equation}
  \limsup_n L(\gamma_n)\leq L(\gamma)\,,
 \end{equation}
i.e., $L$ is upper semicontinuous.
\end{prop}
\begin{pr}
 Let $\sigma=(a=t_0<t_1<\ldots<t_N = b)$ be a partition of $[a,b]$, then the map $\Phi_\sigma$ defined on the space of 
causal curves $\lambda=(\alpha,\beta)\colon[a,b]\rightarrow Y$ given by
 \begin{equation}
  \Phi_\sigma(\lambda):=\sum_{i=0}^{N-1} \sqrt{(\alpha(t_{i+1})-\alpha(t_i))^2 - 
m_{\alpha(t_i),\alpha(t_{i+1})}^2 d(\beta(t_i),\beta(t_{i+1}))^2}
 \end{equation}
is clearly continuous with respect to pointwise convergence. Then $\lv(\lambda) = \inf_{\sigma} \Phi_\sigma(\lambda)$ and so 
$L=\lv$ (by Proposition \ref{prop-l-lv}) is upper semicontinuous as the infimum of continuous functions (cf., e.g., 
\cite[Lem.\ 2.41]{AB:06}).
\end{pr}

\begin{thm}(Limit curve theorem)\label{thm-lim-cur}
Let $(X,d)$ be a metric space and let $\gamma_n=(\alpha_n,\beta_n)$ ($n\in \N$), $\gamma =(\alpha,\beta)\colon[a,b]\rightarrow Y$  be 
absolutely continuous curves such that each 
$\gamma_n$ is future/past directed causal. Moreover, let $\dot\alpha\neq0$ almost everywhere and let 
$\gamma_n\to\gamma$ pointwise. Then $\gamma$ is causal.
\end{thm}
\begin{pr}
 Let $a\leq s<t\leq b$ be such that $\dot\alpha(s)$ and $v_\beta(s)$ exist. For every $n\in\N$ we have by Lemma 
\ref{lem-der-var-geq-0}
 \begin{align}
  (\alpha_n(t)-\alpha_n(s))^2 - m_{\alpha_n(s),\alpha_n(t)}^2 d(\beta_n(s),\beta_n(t))^2 \geq 0.
 \end{align}
Taking the limit $n\to\infty$ yields
 \begin{align}
  (\alpha(t)-\alpha(s))^2 - m_{\alpha(s),\alpha(t)}^2 d(\beta(s),\beta(t))^2 \geq 0\,,
 \end{align}
so 
 \begin{align}
  \Bigl(\frac{\alpha(t)-\alpha(s)}{t-s}\Bigr)^2 - m_{\alpha(s),\alpha(t)}^2 \frac{d(\beta(s),\beta(t))^2}{(t-s)^2} \geq 
0\,.
 \end{align}
Now letting $t\searrow s$ we get
 \begin{align}
\dot\alpha(s)^2 - f(\alpha(s))^2 v_\beta(s)^2\geq 0\,.
 \end{align}
 Moreover, similarly one shows that $\dot\alpha\geq 0$ (if each $\gamma_n$ is future directed) or $\dot\alpha\leq 0$ (if each 
$\gamma_n$ is past directed), which yields $\dot\alpha>0$ or $\dot\alpha<0$ almost everywhere. Thus $\gamma$ is a future or 
past directed causal curve.
 \end{pr}
 
At this point we define a natural time separation function on $Y$, which directly generalizes the spacetime case. Some of 
the results could have been obtained in an even more general setting as they follow just from the existence of a causal 
structure and a length functional --- a fact that was already indicated in \cite[Rem.\ 5.11(i)]{KS:18}. For the interested 
reader we sketch this approach in Appendix \ref{app}, but it is not needed in the following.

\begin{defi}(Time separation function)\label{def-tau}
The time separation function (or Lorentzian distance) $\tau\colon Y\times Y \rightarrow [0,\infty]$ is defined as
\begin{equation}
 \tau(y,y'):=\sup\{L(\gamma):\gamma \text{ future directed causal curve from } y \text{ to } y'\}\,,
\end{equation}
if this set is non-empty, and $\tau(y,y'):= 0$ otherwise.
\end{defi}

\begin{defi}(Causal relations)
 Let $y, y'\in Y$, then $y$ and $y'$ are \emph{chronologically} related, denoted by $y\ll y'$, if there exists a future 
directed timelike curve from $y$ to $y'$. Moreover, $y$ and $y'$ are \emph{causally} related, denoted by $y\leq y'$ 
if there exists a future directed causal curve from $y$ to $y'$ or $y=y'$.

Moreover, we define the chronological and causal future and past of a point as
\begin{align}
&I^+(y):=\{y'\in Y: y\ll y'\}\,,\qquad  I^-(y):=\{y'\in Y: y'\ll y\}\,,\\
&J^+(y):=\{y'\in Y: y\leq y'\}\,,\qquad  J^-(y):=\{y'\in Y: y'\leq y\}\,.
\end{align}
\end{defi}

\begin{lem}\label{lem-cau-rel}
 The relations $\ll$ and $\leq$ are transitive, $\leq$ is reflexive and $\ll\,\subseteq\,\leq$.
\end{lem}
\begin{pr}
Transitivity follows by concatenating curves. 
Reflexivity of the causal relation $\le$ as well as the fact that every timelike curve
is causal hold by definition. Thus $\ll\,\subseteq\,\leq$.
\end{pr}

This can be summarized as:
\begin{rem}\label{rem-tau-pro}
 The time separation function $\tau$ has the following properties:
 \begin{enumerate}
  \item[(i)] $\tau(y,y')=0$ if $y'\not\leq y$ and
  \item[(ii)] $\tau(y,y')>0$ if $y\ll y'$.
 \end{enumerate}
\end{rem}

\begin{lem}(Reverse triangle inequality)\label{lem-rev-tri}
 Let $y_1,y_2,y_3\in Y$ with $y_1\leq y_2\leq y_3$, then
 \begin{equation}
  \tau(y_1,y_2) + \tau(y_2,y_3)\leq \tau(y_1,y_3)\,.
 \end{equation}
\end{lem}
\begin{pr}
This follows from the standard proof from Lorentzian geometry: Let $y_1,y_2,y_3\in Y$ with $y_1\leq y_2\leq y_3$ and assume first that there are future directed causal curves from $y_1$ 
to $y_2$ and from $y_2$ to $y_3$. Then, given $\eps>0$ we can find future directed causal curves $\gamma_1$ from $y_1$ to 
$y_2$ and $\gamma_2$ from $y_2$ to $y_3$ such that $L(\gamma_i) > \tau(y_i,y_{i+1}) - \frac{\eps}{2}$ for $i=1,2$. 
Consequently,
\begin{equation}
 \tau(y_1,y_2) + \tau(y_2,y_3) < L(\gamma_1) + L(\gamma_2) + \eps \leq \tau(y_1,y_3) + \eps\,,
\end{equation}
as the concatenation of $\gamma_1$ and $\gamma_2$ is a future directed causal curve from $y_1$ to $y_3$. Since
$\eps>0$ was arbitrary the claim follows. In the remaining case where there are no future directed causal curves, say from 
$y_1$ to $y_2$, we have $\tau(y_1,y_2)=0$ and $y_1=y_2$, which implies the claim.
\end{pr}

Spacetimes of low regularity (below Lipschitz) can exhibit the unwanted phenomenon of \emph{causal bubbling}, as shown in 
\cite{CG:12}(cf.\ also \cite{GKSS:20}) for spacetimes with continuous metrics. However, the additional structure of a \wpd\ 
excludes such pathologies.

For the formulation of the following result, we recall some terminology from \cite{CG:12}: $Y$
is said to possess the push-up property if the following holds: Whenever $\gamma \colon [a,b]\to M$
is a future/past directed causal curve from $p=\gamma(a)$ to $q=\gamma(b)$ with $L(\gamma)>0$, 
there exists a future/past directed timelike curve connecting $p$ and $q$.

\begin{prop}\label{prop-cau-pla}(Push-up and openness of $I^\pm$)
Every \wpd\ $Y=I\times_f X$ such that $(X,d)$ is a length space has the property that $p\ll q$ if and only if there 
exists a future directed causal curve from $p$ to $q$ of positive length, i.e., push-up holds. 
Moreover, $I^\pm(p)$ is open for any  $p\in Y$. 
\end{prop} 
\begin{pr} For each $p_0 \in I\equiv (a,b)$ we define the function $h_{p_0}\colon(a_{p_0},b_{p_0})\rightarrow (a,b) $ as the 
unique maximal solution of the ODE $\frac{d}{ds}h_{p_0}= f \circ h_{p_0}$ with $h_{p_0}(0)=p_0 $ on $I$. Here $a_{p_0} 
=\int_{p_0}^a \frac{1}{f(s)} ds$ and 
$b_{p_0} =\int_{p_0}^b \frac{1}{f(s)} ds$, and $h_{p_0}$ is the inverse of $r\mapsto \int_{p_0}^r\frac{1}{f(s)}ds$. The function $h_{p_0}$ is strictly increasing, bijective and $C^1$. We 
are going to show that
\begin{equation} \label{I+}
I^+((p_0,\bar{p}))=\{(q_0,\bar{q})\in Y :\, d(\bar{p},\bar{q})< b_{p_0}\,\text{and}\,q_0> h_{p_0}(d(\bar{p},\bar{q}))\},
\end{equation}
which is clearly open. In proving this we will also see that $q\in I^+(p)$ if there exists a causal curve $\gamma $ from $p$ 
to $q$ with $L(\gamma )>0$ (i.e., push-up holds).

We first show that $A(p):=\{(q_0,\bar{q})\in Y :\,  d(\bar{p},\bar{q})< b_{p_0}\,\text{and}\, q_0> 
h_{p_0}(d(\bar{p},\bar{q}))\} \subseteq I^+(p)$. Let $q\in A(p)$ and pick  
an almost minimizing unit-speed curve $\beta\colon[0,d(\bar{p},\bar{q})+\eps]\rightarrow X $ 
from $\bar{p}$ to $\bar{q}$ in $X$, as well as $c>0$ such that $d(\bar{p},\bar{q})+\eps+c<b_{p_0}$ and 
$q_0=h_{p_0}(d(\bar{p},\bar{q})+\eps+c)$. We define $\alpha(s):=h_{p_0}(s+\frac{c}{d(\bar{p},\bar{q})+\eps}s)$. Then $\gamma =(\alpha 
,\beta )$ is a future directed timelike curve from $p$ to $q$, since 
\begin{align}
\dot{\alpha}(s)&=\Big(1+\frac{c}{d(\bar{p},\bar{q})+\eps}\Big)\dot{h}_{p_0}\Big(s+\frac{c}{d(\bar{p},\bar{q})+\eps}s\Big)\\
&>
\dot{h}_{p_0}\Big(s+\frac{c}{d(\bar{p},\bar{q})+\eps}s\Big)
=f\Big(h_{p_0}\Big(s+\frac{c}{d(\bar{p},\bar{q})+\eps}s\Big)\Big)=f(\alpha(s)).
\end{align}

Now we show that if $q\notin A(p)$, then there cannot exist any future directed causal curve $\gamma$ from $p$ to $q$ with 
$L(\gamma )>0$.  Assume to the contrary that such a curve exists and is parametrized so that 
$\gamma\colon[p_0,q_0]\rightarrow Y$ and $\gamma(s)=(s,\bar{\gamma}(s))$. We start with the case where 
$d(\bar{p},\bar{q})<b_{p_0}$ but $q_0\leq h_{p_0}(d(\bar{p},\bar{q}))$. 
Let $\beta_\eps:[0,d(\bar{p},\bar{q})+\eps]\to X$ ($\eps>0$) be an almost minimizing 
unit-speed curve  in $X$ from $\bar{p}$ to $\bar{q}$ and set $\tilde{\beta}_\eps:=\beta_\eps \circ 
h_{p_0}^{-1}\rvert_{[p_0,h_{p_0}(d(\bar{p},\bar{q})+\eps)]}$ and $n_\eps(s):=(s,\tilde{\beta}_\eps(s))$. Then $n_\eps \colon 
[p_0,h_{p_0}(d(\bar{p},\bar{q})+\eps)]\rightarrow Y$ is a null curve, or equivalently, $v_{\tilde{\beta}_\eps}(s)=\frac{1}{f(s)}$. Thus, we have 
\begin{align}
L^d(\tilde{\beta}_\eps )=\int_{p_0}^{h_{p_0}(d(\bar{p},\bar{q})+\eps)} v_{\tilde{\beta}_\eps} =\int_{p_0}^{h_{p_0}(d(\bar{p},\bar{q})+\eps)} \frac{1}{f}.
\end{align}
So letting $\eps \to 0$ gives $d(\bar{p},\bar{q})=\int_{p_0}^{h_{p_0}(d(\bar{p},\bar{q}))} \frac{1}{f} \geq \int_{p_0}^{q_0} \frac{1}{f}$.
Since  $\gamma$ is causal, we have $v_{\bar{\gamma}}\leq \frac{1}{f}$ and furthermore since $L(\gamma)>0$ it must be strictly less than $\frac{1}{f}$ on some subset of $[p_0,q_0]$ having 
non-zero measure. So, 
\begin{align}
d(\bar{p},\bar{q})=\int_{p_0}^{h_{p_0}(d(\bar{p},\bar{q}))} \frac{1}{f} \geq \int_{p_0}^{q_0} \frac{1}{f}>  \int_{p_0}^{q_0} v_{\bar{\gamma}} =L^d(\bar{\gamma}),
\end{align}
a contradiction.

Finally, we treat the case where $d(\bar{p},\bar{q})\geq b_{p_0}$. We again assume that $\gamma\colon[p_0,q_0]\rightarrow Y$ with parametrization $\gamma(s)=(s,\bar{\gamma}(s))$ is a future directed causal curve from $p$ to $q$. 
Since $q_0<b$ and $h_{p_0}(s)\to b$ as $s\nearrow b_{p_0}$ we can choose $\eps >0$ such that $q_0<h_{p_0}(b_{p_0}-\eps )$. Let further 
$x:=\gamma(x_0)=(x_0,\bx)$ be the point on $\gamma $ such that $b_{p_0}-\eps = d(\bar{p},\bar{x})$. So $\gamma|_{[0,x_0]}$ is 
a causal curve from $p$ to $(x_0,\bx)$ with $d(\bar{p},\bar{x})<b_{p_0}$ and $x_0<q_0<h_{p_0}(b_{p_0}-\eps )= 
h_{p_0}(d(\bar{p},\bar{x}))$. Hence, as above, 
\begin{align}
d(\bar{p},\bar{x})=\int_{p_0}^{h_{p_0}(d(\bar{p},\bar{x}))} \frac{1}{f} > \int_{p_0}^{q_0} \frac{1}{f}\geq  \int_{p_0}^{q_0} v_{\bar{\gamma}} =L^d(\bar{\gamma})
\end{align}
 leads to a contradiction. 
\end{pr}

\begin{rem}\label{rem-cau-pla}
The preceding result can be understood as establishing that generalized cones are \emph{causally plain}, i.e., there is no causal bubbling. This notion of causal plainness is however not the same as the one in \cite[Def.\ 1.16]{CG:12} for spacetimes with continuous metrics. The reason is that we cannot speak about approximating smooth metrics, and hence have no notion of timelike curves 
for approximating metrics which have their lightcones inside those of the original Lorentzian metric. However, as shown in \cite{GKSS:20} our notion 
of causal plainness (i.e., the condition that push-up holds) is equivalent to the absence of \emph{external bubbling}, cf.\ \cite[Thm.\ 2.12]{GKSS:20}. Furthermore, if $X$ is a Riemannian manifold with continuous metric $h$ (i.e., if $Y$ is a 
Lorentzian manifold with continuous metric $g=-\dd t^2 +f(t)^2 h$) it can be seen from the description \eqref{I+} of $I^+$ that $I^+=\check{I}^+$ and locally $\partial I^+=\partial J^+$, so that in this case $Y$ is indeed causally plain as defined in 
\cite[Def.\ 1.16]{CG:12}. Moreover, the preceding result also sheds some light on the causality of the so-called \emph{Colombini-Spagnolo metrics} (cf.\ \cite[Sec.\ 2.1]{CG:12}, \cite{CS:89}), i.e., metrics on $\R\times S^1$ of the form $-\dd t^2 + f(t,x)\dd x^2$, where $f(t,x)=\frac{F(t)}{F(x)}$, for a specific continuous positive function $F$.
\end{rem}

\begin{cor} \label{cor-J+-closed} Let $(X,d)$ be a length space. The following description, analogous to \eqref{I+}, holds for $J^+$:
\begin{equation} \label{J+}
\begin{split}
 J^+((p_0,\bar{p}))&=I^+((p_0,\bar{p}))\, \cup \, \{(q_0,\bar{q})\in Y :\,\exists \, \text{a minimizing curve in} \\ 
 & \hspace*{2em} X\, \text{from}\ \bar{p}\,\, \text{to}\,\, \bar{q}\ \text{and}\ d(\bar{p},\bar{q})< b_{p_0}\,\text{and}\ q_0= h_{p_0}(d(\bar{p},\bar{q}))\}.
\end{split}
\end{equation}
Further, if $X$ 
is geodesic, then $J^\pm (p)$ is closed.
\end{cor}
\begin{proof}
That $J^\pm(p)$ is closed if $X$ is geodesic follows from \eqref{J+}: Let $q_k=(q_{0k},\bar q_k)$ be elements
of the right hand side of \eqref{J+} that converge to $q=(q_0,\bar q)$. To see that also $q$ is an element of this
set it suffices to exclude the case where $d(\bar p,\bar q_k)\to b_{p_0} = d(\bar p,\bar q)$. However,
in this case we would have $q_{0k} = h_{p_0}(d(\bar p,\bar q_k))\to b$, resulting in $q_0=b$, a contradiction to
$q_0\in I = (a,b)$. 

It remains to show \eqref{J+}.
First, let $q=(q_0,\bar q)$ be an element of the right hand side of \eqref{J+}, let $\beta:[0,d(\bar p,\bar q)]\to X$ be a
minimizing unit speed curve and set $\gamma: [0,d(\bar p,\bar q)] \to Y$, $\gamma(s):=(h_{p_0}(s),\beta(s))$. Then 
$\gamma$ is a null curve connecting  $p$ and $q$, so $q\in J^+(p)$.

Conversely, we have to show that for any $(q_0,\bar{q})\in J^+(p)\setminus I^+(p)$ 
there must exist a minimizing curve in $X$ from $\bar{p}$ to $\bar{q}$ with 
$d(\bar{p},\bar{q})<b_{p_0}$ and $q_0=h_{p_0}(d(\bar{p},\bar{q}))$. 
Let $\gamma =(\alpha, 
\beta)$ be a causal curve from $p$ to $q$ with $\beta:[0,d]\to X$ parametrized by arc-length. 
By Proposition \ref{prop-cau-pla}, since $q\notin I^+(p)$, we must have $L(\gamma)=0$, i.e., $
\dot{\alpha}^2=f^2v_\beta^2=f^2$ a.e., so $\dot \alpha = f\circ \alpha$. Since $\alpha(0)=p_0$
it follows that $\alpha = h_{p_0}$.  If $\beta$ were not minimizing, i.e., if $d(\bar p,\bar q)<d$,
there would exist a curve $\bar{\beta}:[0,d]\to X$ from $\bar{p}$ to $\bar{q}$, parametrized proportional to 
arclength, which is strictly shorter than
$\beta$, hence satisfies $v_{\bar{\beta}} < 1=v_\beta $ a.e. Then 
$\bar{\gamma}:=(\alpha, \bar{\beta})$ is a timelike curve from $p$ to $q$, contradicting the fact that 
$q\notin I^+(p)$. Consequently, $\beta $ must be minimizing. Thus $d=d(\bar p,\bar q)$, so that 
$q_0=\alpha(d(\bar{p},\bar{q}))=h_{p_0}(d(\bar{p},\bar{q}))$.
\end{proof}

\begin{lem}\label{lem-tau-lsc}
  Let $Y=I\times_f X$ be a \wpd, where $(X,d)$ is a  
  length space. Then the time separation function $\tau$ is 
lower semi-continuous (with respect to $D$).
\end{lem}
\begin{pr}
As the standard proof from Lorentzian geometry only uses openness of $I^+$ and the reverse triangle inequality it still works in our setting: Let $y,y'\in Y$ and first assume that $0<\tau(y,y')<\infty$ (in the case $\tau(y,y')=0$ there is nothing to show). Let 
$0<\eps<\tau(y,y')$, then by definition of $\tau$ there exists a future directed causal curve $\gamma\colon[a,b]\rightarrow 
Y$ from $y$ to $y'$ with $L(\gamma)\geq \tau(y,y')-\frac{\eps}{2}>0$. By Remark \ref{rem-len-int-ac} there are $0<t_1\leq 
t_2<b$ such that $0<L(\gamma\rvert_{[a,t_1]})<\frac{\eps}{4}$ and $0<L(\gamma\rvert_{[t_2,b]})<\frac{\eps}{4}$. Setting 
$y_1:=\gamma(t_1)$, $y_2:=\gamma(t_2)$ and $U:=I^-(y_1)$, $V:=I^+(y_2)$ we obtain that $\tau(y,y_1)\geq 
L(\gamma\rvert_{[a,t_1]})>0$, thus $y\ll y_1$ and hence $y\in U$, which by Proposition \ref{prop-cau-pla} is 
an open neighborhood of $y$. Analogously we get 
that $y'$ is in the open set $V$. At this point let $(r,r')\in U\times V$, then $r\ll y_1\leq y_2\ll r'$ and thus by the 
reverse triangle inequality (Lemma \ref{lem-rev-tri}) we obtain
\begin{align}
 \tau(r,r')&\geq \underbrace{\tau(r,y_1)}_{\geq 0} + \tau(y_1,y_2) + \underbrace{\tau(y_2,r')}_{\geq 0} \geq \tau(y_1,y_2) 
\geq L(\gamma\rvert_{[t_1,t_2]})\\
  &= L(\gamma) - L(\gamma\rvert_{[a,t_1]}) - L(\gamma\rvert_{[t_2,b]}) \\
  &\geq \tau(y,y') - 
\frac{\eps}{2} - \frac{\eps}{4} - \frac{\eps}{4} = \tau(y,y') - \eps\,,
 \end{align}
which finishes this case. For the case $\tau(y,y')=\infty$ the above construction shows the existence of arbitrarily long 
future directed causal curves from $r$ to $r'$, so $\tau$ attains arbitrarily large values on suitable neighborhoods of $(y,y')$.
\end{pr}

\begin{prop}\label{prop-wpd-lpls}
  Let $Y=I\times_f X$ be a \wpd, where $(X,d)$ is a  
  length space. Then $(Y,D,\ll,\leq,\tau)$ is a \LpLSn.
\end{prop}
\begin{pr}
 By Lemma \ref{lem-cau-rel} $(Y,\ll,\leq)$ is a causal space and by Lemma \ref{lem-tau-lsc} the time separation function is 
lower semi-continuous. Finally, Remark \ref{rem-tau-pro} and Proposition \ref{prop-cau-pla} give the required properties of 
$\tau$, cf.\ \cite[Def.\ 2.8]{KS:18}. 
\end{pr}

\begin{ex}\label{ex-warped-product-mf} Let $f:I\to (0,\infty)$ be continuous and let $(X,h)$ be a Riemannian manifold.
Then considered as a Lorentzian pre-length space, the warped product $I\times_f X$, i.e.,
the product manifold $I\times X$ endowed with the continuous Lorentzian metric
$-dt^2 + f^2 h$ coincides with the generalized cone $I\times_f X$, as is immediate
from Definitions \ref{def-len} and \ref{def-tau}. Therefore, there is no ambiguity
in our notation.
\end{ex}

\begin{defi}
 Let $Y=I\times_fX$ be a \wpd\ and let $\gamma=(\alpha,\beta)\colon[a,b]\rightarrow Y$ be a causal curve. Then the \emph{energy} of 
$\gamma$ is defined as
\begin{equation}
 E(\gamma):=\frac{1}{2}\int_a^b \dot\alpha^2 -(f\circ\alpha)^2 v_\beta^2\,.
\end{equation}
\end{defi}
Contrary to the length, the energy of a curve depends on its parametrization. Nevertheless it will turn out to be a useful tool.

The following is an analogue of \cite[Thm.\ 3.1]{AB:98} in the Riemannian case.
\begin{thm} \label{thm-structure-of-geod}
 Let $(X,d)$ be a geodesic length space and let $\gamma=(\alpha,\beta)\colon[0,b]$ $\rightarrow Y = I\times_f X$ be future 
directed causal and maximal. Then:
 \begin{enumerate}[label=(\roman*)] 
  \item The fiber component $\beta$ is minimizing in $X$.
  \item\label{thm-structure-of-geod-fib-ind} Fiber independence holds, i.e., the base component $\alpha$ is independent of $\beta$, i.e., $\alpha$ depends only on the length of $\beta$. More precisely, let $(X',d')$ be another geodesic length space, $\beta'$ minimizing in $X'$ with 
$L^{d'}(\beta')=L^d(\beta)$ and the same speed as $\beta$, i.e, $v_\beta=v_{\beta'}$. Then $\gamma':=(\alpha,\beta')$ is a 
future directed maximal causal curve in $Y':=I\times_f X'$, which is timelike if $\gamma$ is timelike in $Y$.
  \item If $\gamma$ is timelike, then it has an (absolutely continuous) parametrization with respect to arclength, i.e., 
$-\dot\alpha^2 + (f\circ\alpha)^2 v_\beta^2 = -1$ almost everywhere. 
 \item If $\gamma$ is timelike and parametrized with respect to arclength (so $b=L(\gamma)$), then the energy of $\gamma$, 
$E(\gamma)$, is minimal under all reparametrizations of $\gamma$ on $[0,b]$.
 \item If $\gamma$ is timelike and parametrized with respect to arclength, then $v_\beta$ is proportional to 
$\frac{1}{(f\circ\alpha)^2}$.
\item If $\gamma$ is timelike, it has an (absolutely continuous) parametrization proportional to arclength such that 
$-\dot\alpha^2 + \frac{1}{(f\circ\alpha)^2}$ is constant.
 \end{enumerate}
\end{thm}
\begin{pr}
 \begin{enumerate}
\item[(i)] Assume that $\beta$ is not minimal (and hence not constant). 
We may suppose that $\beta$ is parametrized with 
respect to arclength, i.e., $v_\beta=1$ almost everywhere and $\gamma\colon[0,b]\rightarrow X$, where $b=L^d(\beta)$. 
Since $\beta $ is not minimal, there exists another curve $\bar \beta $ from $\beta(0)$ to 
$\beta(b)$, parametrized with respect to arclength, with $0< L^d(\bar\beta) < L^d(\beta)= b$. We set $T:= L^d(\bar\beta) < b $ and define $\bar \gamma = (\bar 
\alpha,\bar\beta)\colon[0,T]\rightarrow Y$ by setting 
$\bar\alpha(s):=\alpha(\frac{b}{T}s)$ for $s\in[0,T]$. Then clearly $\bar\gamma$ is timelike and using the reparametrization 
$\bar s = \frac{b}{T} s$ we get
\begin{align}
 L(\bar\gamma) &= \int_0^T \sqrt{(\dot{\bar\alpha}(s))^2 - f(\bar\alpha(s))^2}\,\dd s\\
  &= \frac{T}{b}\int_0^b \sqrt{\Big(\frac{
 b}{T}\dot{\alpha}(\bar s)\Big)^2 - f(\alpha(\bar s))^2}\,\dd \bar s \\
 &= \int_0^b \sqrt{\dot{\alpha}(\bar s)^2 - \Big(\frac{T}{b}f(\alpha(\bar s))\Big)^2}\,\dd \bar s
> L(\gamma)\,,
\end{align}
a contradiction.
\item[(ii)] Let $\beta'$ be minimizing in $X'$, defined on $[0,b]$ with $L^{d'}(\beta')=L^d(\beta)$ and $v_\beta=v_{\beta'}$. Set 
$\gamma':=(\alpha,\beta')\colon[0,b]\rightarrow Y'$. Then  $\gamma'$ is future directed causal and
\begin{align}
 L(\gamma') = \int_0^b\sqrt{\dot\alpha^2 - (f\circ\alpha)^2 v_{\beta'}^2} = \int_0^b\sqrt{\dot\alpha^2 - (f\circ\alpha)^2 
v_{\beta}^2} = L(\gamma)\,. 
\end{align}
It remains to show that $\gamma'$ is maximal in $Y'$ from $(\alpha(0),\beta'(0))=:(t_0,x'_0)=:y'_0$ to 
$(\alpha(b),\beta'(b))=:(t_1,x'_1)=:y_1'$. To this end assume to the contrary that there is a 
$\tilde\gamma=(\tilde\alpha,\tilde\beta)\colon[0,b]\rightarrow Y'$ that is future directed causal from $y'_0$ to $y'_1$ and 
longer than $\gamma'$, i.e., $L(\tilde\gamma)>L(\gamma')$. Without loss of generality we may assume that $\tilde{\gamma}$ is parametrized such that $\tilde\beta$ has speed $v_{\tilde{\beta}}$ proportional to $v_{\beta'}$. 
By minimality of $\beta'$ this implies that $v_{\tilde\beta}\geq  v_{\beta'}$. Set 
$\bar\gamma:=(\tilde\alpha,\beta')\colon[0,b]\rightarrow Y'$, then $\bar\gamma$ is future directed causal 
from $y_0'$ to $y_1'$ and $L(\tilde\gamma)\leq L(\bar\gamma)$. Furthermore, we obtain
\begin{align}
L(\bar\gamma) = \int_0^b\sqrt{\dot{\tilde\alpha}^2 - (f\circ\tilde\alpha)^2 v_{\beta'}^2} = 
\int_0^b\sqrt{\dot{\tilde\alpha}^2 - (f\circ\tilde\alpha)^2 v_{\beta}^2} = L((\tilde\alpha,\beta))\,.
\end{align}
Consequently, $L((\tilde\alpha,\beta))= L(\bar\gamma)\geq L(\tilde\gamma) > L(\gamma') = L(\gamma)$, contradicting the 
maximality of $\gamma$, as $\tilde\alpha$ and $\alpha$ have the same endpoints.
\item[(iii)] Let $\gamma$ be timelike and define $\phi(s):=\int_0^s\sqrt{\dot\alpha^2 - (f\circ\alpha)^2 v_\beta^2}$ for 
$s\in[0,b]$. Then $\phi\colon[0,b]\rightarrow[0,L(\gamma)]$ is absolutely continuous and strictly monotonically increasing. 
Moreover, $\phi^{-1}$ exists and is absolutely continuous as $\dot\phi>0$ almost everywhere, (\cite[p.\ 271]{Nat:55}. 
Consequently, $\tilde\gamma:=\gamma\circ\phi^{-1}\equiv (\tilde \alpha,\tilde\beta)$ 
is absolutely continuous by~\cite[Thm.\ 3, Ch.\ IX]{Nat:55} and 
satisfies $-\dot{\tilde\alpha}^2 + (f\circ\tilde\alpha)^2 v_{\tilde\beta}^2 = -1$ almost everywhere. 

\item[(iv)] This follows directly from the Cauchy-Schwartz inequality applied to the length of any reparametrization of $\gamma$ on 
$[0,b]$.

\item[(v)] The claim follows as in the proof of \cite[Thm.\ 3.1]{AB:98} by establishing that $\int_I (f\circ\tilde\alpha)^2 
v_{\tilde\beta} = \int_J (f\circ\tilde\alpha)^2 v_{\tilde\beta}$ for all intervals $I,J\subseteq [0,b]$ of the same length, 
where one uses point (iv) above.

\item[(vi)] By the previous two points we can assume that $\gamma$ is parametrized with respect to arclength, i.e., 
$-\dot{\alpha}^2 + (f\circ\alpha)^2 v_{\beta}^2 = -1$ almost everywhere and that $v_\beta=\frac{c}{(f\circ\alpha)^2}$ for 
some constant $c$. For $c\neq 0$, the reparametrization $\tilde\gamma(s):=\gamma(\frac{s}{c})$ does the job, and for $c=0$ 
(i.e., $v_\beta=0$) the reparametrization $\tilde\gamma=(\tilde\alpha,\tilde\beta):=\gamma\circ\phi^{-1}$, where $\phi(t):=\int_0^t f\circ\alpha$ yields $-\dot{\tilde\alpha}^2 + \frac{1}{(f\circ\tilde\alpha)^2} =0$.
\end{enumerate}
\end{pr}

As a first consequence of fiber-independence we obtain:
\begin{cor}\label{cor-f-lip} Let $X$ be a geodesic length space and let $Y=I\times_f X$. Then any maximizing causal curve $\gamma = (\al,\beta): [-b,b]\to Y$ has a causal character,
i.e., $\gamma$ is either timelike or null.
\end{cor}
\begin{pr} Denote by $Y'$ the Lorentzian warped product $I\times_f \R$, i.e., the manifold
$I\times \R$ endowed with the continuous metric $-dt^2 + f^2 dx^2$. 
Let $\beta$ be parametrized by arclength and set $\beta': [-b,b]\to \R$, $\beta'(t):=t$.
Then by Theorem \ref{thm-structure-of-geod},\ref{thm-structure-of-geod-fib-ind} (and Example \ref{ex-warped-product-mf}), 
$\gamma':=(\alpha,\beta')$
is a causal maximizer in $Y'$. 
We now use the same basic ideas as in the proof of \cite[Thm.\ 1.1]{GL:18} (the difference being that the construction of the relevant curves is different, due to the metric being not locally Lipschitz but having a warped product structure) to show that $\gamma'$ is either timelike or null. The same must therefore be true for $\gamma$.

Since we exclusively work in $Y'$ from now on, we will drop the $'$ from our notation. Assume $\gamma $ is neither null nor timelike. Without loss of generality, we may assume $\gamma(0)=(\alpha(0),0)=(0,0)$, $\dot{\gamma}(0)$ exists and is timelike and $N:=\{s\in [-b,0]:\,\dot{\gamma}(s) \text{ is null} \}$ has non-zero measure. Let $\eps $ be positive and define $\gamma_\eps : [-b,0]\to I\times \R$ by \[\gamma_\eps (s)=(\alpha(s),\beta_\eps(s)):=(\alpha(s),\sqrt{1-\eps}\,s+b\,\sqrt{1-\eps}-b).\]
Then $\gamma_{\eps}(-b)=\gamma(-b)=(\alpha(-b),-b)$ and $\gamma_\eps(0)=(\alpha(0),\beta_\eps(0))=(0,b\,\sqrt{1-\eps}-b)$. Note that for $\eps<\frac{1}{2}$, there exists $C>0$ such that $|\beta_\eps(0)|=|b\,\sqrt{1-\eps}-b| \leq C \eps$. We estimate
\begin{align} L(\gamma_\eps)&=\int_{-b}^0\sqrt{\dot{\alpha}^2-(f\circ\alpha)^2 (1-\eps)}=\int_{[-b,0]\setminus N} \sqrt{\dot{\alpha}^2-(f\circ\alpha)^2 (1-\eps)} +\\
&+\sqrt{\eps} \underbrace{\int_{N}f\circ\alpha}_{=:c>0} \geq \int_{[-b,0]\setminus N} \sqrt{\dot{\alpha}^2-(f\circ\alpha)^2}+ c \sqrt{\eps}=L(\gamma|_{[-b,0]})+c \sqrt{\eps}.
\end{align}
Next, note that there exist $\eta_0>0$ and $C_0>0$ such that $f(t)<C_0<\dot{\alpha}(0)$ for $t\in [-\eta_0,\eta_0]$ and since $\alpha(s)=(\dot{\alpha}(0)+h(s))s$ with $h(s)\to 0$ as $s\to 0$, there also exist $0<\eta_1$ and $0<C_0<C_1<\dot{\alpha}(0)<C_2$ such that $C_2 s> \alpha(s)>C_1 s$ for $s\in [0,\eta_1]$. Fix $k$ with $C_0<k<C_1$ and set $s_\eps:=\frac{-k}{C_1-k} \beta_\eps(0)$. Then $0<s_\eps < \frac{Ck}{C_1-k} \eps$ for $\eps <\frac{1}{2}$. Let $\eps $ small enough such that $s_\eps+|\beta_\eps(0)|<\min\{\eta_1,\frac{\eta_0}{C_1}\}$. Then the straight lines $h : s\mapsto (C_1 s,s)$ and $g_k : s\mapsto (ks-k\beta_\eps(0),s)$ are timelike on $[\beta_\eps(0),s_\eps]$ and intersect each other in $s_\eps>0$. Further, $\alpha(s)>C_1 s$ on $(0,s_\eps]$, so $(\alpha(0),0)=(0,0)$ lies strictly below $g_k$ but $(\alpha(s_\eps),s_\eps)$ lies strictly above $g_k$ (since it lies strictly above $h(s_\eps)$ which is equal to $g_k(s_\eps)$) and hence $s\mapsto (\alpha(s),s)=\gamma(s) $ intersects $g_k$ in some $0<\bar{s}< s_\eps$. Note that $g_k|_{[\beta_\eps(0),\bar{s}]}$ is a future directed timelike curve from $\gamma_\eps(0)$ to $\gamma(\bar{s})$. Now we estimate the length of the concatenation as follows:
\begin{align} L(\gamma_\eps*g_k|_{[\beta_\eps(0),\bar{s}]})&> L(\gamma_\eps)\geq c \sqrt{\eps}+L(\gamma|_{[-b,0]})\\
&=L(\gamma|_{[-b,\bar{s}]})+c \sqrt{\eps}-L(\gamma|_{[0,\bar{s}]}) \\
&\geq L(\gamma|_{[-b,\bar{s}]})+c \sqrt{\eps}- \alpha(\bar{s})\\
&> L(\gamma|_{[-b,\bar{s}]})+c \sqrt{\eps}- C_2 \bar{s} \\
&> L(\gamma|_{[-b,\bar{s}]})+c \sqrt{\eps}- C_2 s_\eps \\
& \ge  L(\gamma|_{[-b,\bar{s}]})+c \sqrt{\eps}- C_2 \frac{Ck}{C_1-k} \eps > L(\gamma|_{[-b,\bar{s}]})
\end{align}
for $\eps$ small. This contradicts the maximality of $\gamma$.
\end{pr}

\begin{ex}\label{ex:conesaregencones}(Minkowski cones as generalized cones.) 
Here we show that Minkowski cones as defined 
in Section \ref{sec-con} can equivalently be viewed as generalized cones. Let $X$ be a geodesic length space, 
let $Y:=\mathrm{Cone}(X)$ be the Minkowski cone over $X$, with relations $\ll_Y, \leq_Y$ and time separation function 
$\tau_Y$. Let $G:=(0,\infty)\times_\mathrm{id} X$ be the generalized cone with warping function $f=\mathrm{id}$ over $X$. 
Since we did not explicitly treat generalized cones of the form $I\times_f X$ with a non-open interval $I$ and a function $f$ 
that might be zero at the endpoints (though, as indicated in Remark~\ref{rem:closedintervals} below, these cases could 
be included relatively straightforwardly), we will compare the Lorentzian pre-length space $(G,D,\ll_G,\leq_G,\tau_G)$ 
with the Lorentzian pre-length space $Y':=Y\setminus \{0\}=(0,\infty)\times X$ with metric $D$
(which is equivalent to the restriction of the cone metric $d_c$), relations 
$\ll_{Y'}:=\ll_Y|_{Y'\times Y'}$, $\leq_{Y'}:=\leq_Y|_{Y'\times Y'}$ and time separation function 
$\tau_{Y'}:=\tau_Y|_{Y'\times Y'}$.
	
Let $x=(x_0,\bx),y=(y_0,\by)\in G$, then by the description of $I^+$ in \eqref{I+}  $x\ll_G y$ if and only 
if for corresponding points $x'=(x_0,\bx')\in W:=(0,\infty)\times_\mathrm{id} \hyp^{n-1} $ and  $y'=(y_0,\by')\in W$
 with $d_X(\bx,\by)=d_{\hyp^{n-1}}(\bx',\by')$ one has $x'\ll_{W} y'$. Similarly from \eqref{J+} we see that, since both $X$ and $\hyp^{n-1}$ are geodesic,  the same holds for $\leq$. Lastly, by fiber independence (Theorem \ref{thm-structure-of-geod},\ref{thm-structure-of-geod-fib-ind}) we also have that $\tau_G(x,y)=\tau_{W}(x',y')$. 
 By the last two paragraphs in Remark \ref{model}, we have $\leq_{W}=\leq_C$, $\ll_{W}=\ll_C$ and $\tau_{W}=\tau_C$ (where $C:=(0,\infty)\times \hyp^{n-1}$ with the Minkowski cone structure as in Remark \ref{model}). 
 Now since by definition $\leq ,\ll,\tau$ for Minkowski cones are clearly fiber independent as well, we have  $x\ll_{Y'} y$  
 if and only if $x'\ll_C y'$ if and only if $x\ll_G y$. The same holds for $\leq $. Also clearly 
 $\tau_{Y'}(x,y)=\tau_C(x',y')=\tau_G(x,y)$. So the Lorentzian pre-length spaces $(G,D,\leq_G,\ll_G,\tau_G)$ and $(Y',D,\leq_{Y'}.\ll_{Y'},\tau_{Y'} )$ can be identified.
\end{ex}
\begin{rem}\label{rem:closedintervals} We have confined ourselves in this section to generalized cones
	$I\times_f X$ with $I$ an open interval, but note that
	general intervals $I$ could be treated in complete analogy. One could also consider the case where $f$ has isolated zeros (either in the interior of the interval or at the interval boundaries) with the additional assumption that the improper integrals of $\frac{1}{f}$ coming from both sides of each zero diverge. 
	If $f(t)=0$ for some $t\in I$, we identify $(t,x)\sim (t,x')$ for $x,x'\in X$ to a point denoted by $t_Y$. Defining all concepts analogously to the case where $f>0$, it is easy to see that for any $(x_0,x)\in Y$ with $x_0<t$ and $x\in X$ arbitrary we have $(x_0,x)\ll t_Y$ and $\gamma: s\mapsto (x_0+s,x)$ is a future directed timelike curve from $(x_0,x)$ to $t_Y$ with $\tau((x_0,x),t_Y)=L(\gamma)$. Further, $I^+(t_Y)=(I\cap (t,\infty))\times X$ and $I^-(t_Y)= (I\cap (-\infty , t))\times X$. So any two points $(x_0,x),(y_0,y)$, $x_0<y_0$ with $f$ having a zero on $[x_0,y_0]$ are trivially timelike related. Therefore considering $f$ having zeros in the interior of $I$ largely reduces to the problem of allowing $f$ to vanish at the boundary. Divergence of the integral of $\frac{1}{f}$ as one approaches the zeroes of $f$ ensures that $I^\pm$ remains open (see Proposition \ref{prop-cau-pla}) and thus, with some modifications, the main results in Sections 
	\ref{sec:gen_cone} and \ref{sec:gen_cones_as_lls} should remain valid, but this would need to be investigated more carefully.
\end{rem}

\section{Generalized cones as \LLSn s} \label{sec:gen_cones_as_lls}
We already established that every \wpd\ $Y=I\times_fX$, where $(X,d)$ is a length space, is a \LpLS in Proposition \ref{prop-wpd-lpls}. 
Here we will show that such spaces are in fact \LLSn s if $X$ is locally compact. To this end we need the following auxiliary results.

\begin{lem}\label{lem-jpq}
	Let $(X,d)$ be a metric space and let $p=(t_0,\bar p)$, $q=(t_1,\bar q)\in Y$, then for the causal diamond $J(p,q):=J^+(p)\cap J^-(q)$ we have
	\begin{align}
	J(p,q)\subseteq \{(t,\bar r)\in Y: t_0\leq t\leq t_1,\ \bar r\in \bar B^d_{\frac{t-t_0}{m_{t_0,t}}}(\bar 
	p)\cap \bar B^d_{\frac{t_1-t}{m_{t,t_1}}}(\bar q)\}\,,
	\end{align}
	where $\bar B^d_\delta(x)=\{x'\in X: d(x,x')\leq \delta\}$ denotes the closed ball of radius $\delta$ in $X$.
\end{lem}
\begin{pr}
	Let $r=(t,\bar r)\in J(p,q)$, $p<r<q$, and let $\gamma=(\alpha,\beta)\colon[0,b]\rightarrow Y$ be a future directed causal curve from $p=\gamma(0)$ to 
	$r=\gamma(t^*)$ to $q=\gamma(b)$. Then $\dot\alpha>0$ almost everywhere. We get
	$t_0=\alpha(0)\leq \alpha(t^*) = t \leq \alpha(b) = t_1$. From the proof of Lemma \ref{lem-der-var-geq-0} we conclude that
	\begin{align}
	t-t_0 = \alpha(t^*)-\alpha(0) \geq m_{t_0,t}\,d(\bar p,\bar r)\,, 
	\end{align}
	and analogously $t_1-t\geq m_{t,t_1}d(\bar r,\bar q)$.
\end{pr}

\begin{lem}\label{lem-causally-convex-nbhds}
	Let $(X,d)$ be a metric space. Then any $p=(p_0,\bar p)\in Y$ has a basis of open, \emph{causally convex} neighborhoods, i.e., neighborhoods such that any causal curve with endpoints in that neighborhood is contained in it. This also shows that such a generalized cone is \emph{strongly causal}. Moreover, the map $Y\rightarrow~I\colon (t,x)\mapsto~t$ is a \emph{time function}, i.e., $t$ is  continuous and strictly increasing along any future directed causal curve.
\end{lem}	
\begin{proof}
	Using the same arguments as in the proof of the previous Lemma one easily checks that the family
	\[ \bigg\{ (t,\bar r)\in Y: p_0-\eps < t< p_0+\eps ,\ \bar r\in  B^d_{\frac{t-(p_0-\eps )}{m_{p_0-\eps ,p_0+\eps }}}(\bar 
	p)\cap B^d_{\frac{p_0+\eps -t}{m_{p_0-\eps ,p_0+\eps }}}(\bar p) \bigg\}_{\eps >0 } \]
	satisfies the claim.
\end{proof}

\begin{lem}\label{lem-D-com}
  Every \wpd\ has the property that for every point $y$ there is a neighborhood $U$ of $y$ 
and a constant $C>0$ such that the (metric) $D$-arclength of every causal curve which is contained in $U$ is bounded by $C$, i.e., 
$L^D(\gamma)\leq C$.
\end{lem}
\begin{pr}
Let $y=(t,p)\in Y$ and let $I'\subseteq I$ be a compact interval containing $t$. Set $C':=\diam(I')$ and $C:=\min_{r\in I'}f(r)>0$. 
Moreover, let $\gamma=(\alpha,\beta)\colon[a,b]\rightarrow Y$ be a (without loss of generality) future directed causal curve that is contained 
in $U:=I'\times X$. Then since $C^2 v_\beta^2\leq (f\circ\alpha)^2 v_\beta^2 \leq \dot\alpha^2$ we get
\begin{align}
 L^D(\gamma) = \int_a^b\sqrt{\dot\alpha^2 + v_\beta^2} \leq \int_a^b\dot\alpha\sqrt{1+\frac{1}{C^2}} \leq \sqrt{1+\frac{1}{C^2}}\,C'\,,
\end{align}
as required.
\end{pr}

We want to establish that \wpds\ are \LLSn s. For this we first need to show that the different notions of causal curves and 
their length agree with the ones in the setting of \LLSn s.

In the following result (and thereafter), when comparing the different notions of causal curves, it will always be 
understood that parametrizations are chosen in which the respective curves are never locally constant
(cf.\ \cite[Ex.\ 2.5.3]{BBI:01}).
\begin{lem}\label{lem-cc-rcc}
 The notion of causal curves for a \wpd\ agrees with the notion of a causal curves with respect to the relation $\leq$ (cf.\ 
\cite[Def.\ 2.18]{KS:18}).
\end{lem}
\begin{pr}
 Clearly, every future or past directed curve in the sense of Definition \ref{def-cc} is causal with respect to $\leq$. For 
the converse, note that since this is a local question it suffices to consider segments of causal curves. So let 
$\gamma=(\alpha,\beta)\colon[a,b]\rightarrow Y$ be a (without loss of generality) future directed causal curve with respect 
to $\leq$, i.e., $\forall a\leq s\leq t \leq b$: $\gamma(s)\leq\gamma(t)$. Thus for any $a\leq s < t\leq b$ there is a future 
directed causal curve (in the sense of Definition \ref{def-cc}) $\gamma_{s,t}=(\alpha_{s,t},\beta_{s,t})\colon[0,1]\rightarrow Y$ from 
$\gamma(s)$ to $\gamma(t)$. This implies that $\alpha$ is strictly monotonically increasing as $t$ is a time function (cf.\ Lemma 
\ref{lem-causally-convex-nbhds}).

We now want to construct a sequence of future directed causal curves (in the sense of Definition \ref{def-cc}) that 
converges pointwise to $\gamma$. For $\sigma:=(a=t_0<t_1,\ldots,t_N=b)$ a partition of $[a,b]$, 
denote by $\gamma^\sigma$ the future 
directed causal curve $\gamma^\sigma:=\gamma_{t_0,t_1}*\ldots*\gamma_{t_{N-1},t_N}$ obtained by concatenating the 
curves $\gamma_{t_i,t_{i+1}}$ ($0\le i\le N-1$).  Let $\sigma_k$ be a sequence of such partitions whose
norms (maximal length of a subinterval) tend to zero as $k\to \infty$. We show that $\gamma_k := \gamma^{\sigma_k}$
converges pointwise to $\gamma$.
Let $t\in[a,b]$ and let $U$ be a neighborhood of $\gamma(t)$. By Lemma \ref{lem-causally-convex-nbhds} there exists a causally convex neighborhood $V$ of $\gamma(t)$ such that $V\subseteq U$. As $\gamma^{-1}(V)$ is a neighborhood of $t$ in $[a,b]$, for $k$ large any sub-interval of $\sigma_k$ containing
$t$ lies entirely in $V$, so in particular $\gamma_k(t)\in U$.
Consequently, $\gamma_k\to\gamma$ pointwise and thus by the limit curve theorem \ref{thm-lim-cur} 
$\gamma$ is a (future directed) causal curve in the sense of Definition \ref{def-cc}.  
\end{pr}

\begin{prop}[Local existence of maximal causal curves] \label{prop-loc-ex-max-cc}
 Let $(X,d)$ be a locally compact metric space. Then every point in $Y=I\times_f X$ has a neighborhood $U$ such that any two causally related points in $U$ can be connected by a maximal causal curve.
\end{prop}
\begin{pr}
Let $p\in Y$,  $U'=I'\times X$ and $C>0$ be given by Lemma \ref{lem-D-com}, let $W\subseteq X$ be a compact neighborhood of $\bar p$ in $X$ and set $V:=I' \times W$. Further, let $U \subseteq V \subseteq U'$ be causally convex in $V$ (cf.\ Lemma \ref{lem-causally-convex-nbhds}). Let $y,z\in U$ with $y<z$, and note that any causal curve from $y$ to $z$ has to be contained in $U$. So local maximality in $U$ implies global maximality. Let $\gamma_n\colon[a,b]\rightarrow Y$ be a sequence of future directed causal curves from $y$ to $z$ such that $L(\gamma_n)\to\tau(y,z)$. Then, by Lemma \ref{lem-D-com}, $L^D(\gamma_n)\leq C$ and so reparametrizing each $\gamma_n$ proportional to $D$-arclength on $[a,b]$ yields a sequence of uniformly $D$-Lipschitz curves $\tilde\gamma_n$ each of which is future directed causal. By the theorem of Arzela-Ascoli (the sequence  $\gamma_n$ is contained in the compact set $V$) we obtain a subsequence $(\tilde\gamma_{n_k})_k$ of $(\tilde\gamma_n)_n$ that convergences uniformly to a Lipschitz curve $\gamma$ from $y$ to $z$. As $y<z$ this curve cannot be constant and so by possibly reparametrizing $\gamma$ such that it is never 
locally constant we obtain a future directed causal curve $\gamma$ from $y$ to $z$ that is contained in $U$. Moreover, by Proposition \ref{prop-L-usc} we get that
\begin{align}
 L(\gamma)\leq\tau(y,z) = \limsup_k L(\gamma_{n_k})\leq L(\gamma)\,,
\end{align}
so $\gamma$ is maximal.
\end{pr}

A similar argument gives that $Y$ is \emph{locally causally closed} (\cite[Def.\ 3.4]{KS:18}):
\begin{lem}\label{lem-loc-cau-clo}
 Let $(X,d)$ be a locally compact metric space. Then every point in $Y$ has a neighborhood $U$ such that for any $y_n,z_n\in Y$ with $y_n\to y\in\bar U$, $z_n\to z\in\bar U$ 
and $y_n\leq z_n$ for all $n\in\N$, it follows that $y\leq z$.
\end{lem}

The next step is to establish that the length of a causal curve agrees with the $\tau$-length introduced in \cite[Def.\ 2.24]{KS:18}.
Recall that the $\tau$-length, $L_\tau(\gamma)$, is defined as
\begin{align}
 L_\tau(\gamma):=\inf\{\sum_{i=0}^{N-1}\tau(\gamma(t_i),\gamma(t_{i+1})): a=t_0<t_1<\ldots<t_N=b\}\,,
\end{align}
where $\gamma$ is a future directed causal curve (and by Lemma \ref{lem-cc-rcc} this is the same as causal with respect to $\leq$).
\begin{prop}\label{prop-L-Ltau}
 Let $(X,d)$ be a locally compact metric space. If $\gamma\colon[a,b]\rightarrow Y$ is  a future directed causal curve, then $L(\gamma)=L_\tau(\gamma)$.
\end{prop}
\begin{pr}
Let $a=t_0<t_1\ldots<t_N=b$ be a partition of $[a,b]$, then
\begin{align}
 \sum_{i=0}^{N-1}\tau(\gamma(t_i),\gamma(t_{i+1}))\geq \sum_{i=0}^{N-1} L(\gamma\rvert_{[t_i,t_{i+1}]}) = L(\gamma)\,,
\end{align}
as $L$ is additive. Taking the infimum over all partitions of $[a,b]$ yields $L(\gamma)\leq L_\tau(\gamma)$.

For the reverse inequality we cover $\gamma([a,b])$ by neighborhoods $U_0,\ldots,U_N$ 
as in Proposition \ref{prop-loc-ex-max-cc} and choose a 
partition $\sigma:=(a=t_0<t_1<\ldots<t_{N+1}=b)$ such that $\gamma(t_{i+1})\in U_i\cap U_{i+1}$ for every $i=0,\ldots,N-1$. Consequently, 
there are future directed maximal causal curves $\gamma_i^\sigma$ from $\gamma(t_i)$ to $\gamma(t_{i+1})$ for $i=0,\ldots,N$. The future 
directed causal curve $\gamma^\sigma:=\gamma_0^\sigma*\ldots*\gamma_N^\sigma$ has length
\begin{align}
 L(\gamma^\sigma) = \sum_{i=0}^N L(\gamma_i^\sigma) = \sum_{i=0}^N \tau(\gamma(t_i),\gamma(t_{i+1})) \geq L_\tau(\gamma)\,.
\end{align}
By shrinking the cover $(U_i)_i$ and adapting the partition $\sigma$ accordingly we get a sequence $\gamma_k$ of future directed causal curves which, by an argument as in the proof of Lemma \ref{lem-cc-rcc}
converges pointwise to $\gamma$, and satisfies $L(\gamma_k)\geq L_\tau(\gamma)$ for all $k\in\N$. As $L$ is upper semicontinuous by 
Proposition \ref{prop-L-usc} we get $L(\gamma)\geq L_\tau(\gamma)$ and this finishes the proof.
\end{pr}
Thus there is no need to distinguish between $L$ and $L_\tau$ and the different notions of causal curves also agree, so when applying the 
theory of \LLSn s to \wpds\ we will always use the notions of the latter.

\begin{thm}\label{thm-wpd-lls}
 Any \wpd\ $I\times_f X$, where $(X,d)$ is a locally compact length space, is a strongly causal \LLSn.
\end{thm}
\begin{pr}
 By Proposition \ref{prop-wpd-lpls} and Lemma \ref{lem-loc-cau-clo} $Y$ is a locally causally closed \LpLSn. Moreover, by 
definition of the causal relations it is \emph{causally path connected}.

Since the different notions of causal curves agree by Lemma \ref{lem-cc-rcc} and as $L_\tau=L$ 
by Proposition \ref{prop-L-Ltau}, we directly obtain that $\tau=\Tau$, where
\begin{align}
 \Tau(y,z)=\sup\{L_\tau(\gamma):\gamma \text{ future-directed causal from }y \text{ to } z\}\,,
\end{align}
 if the set is non-empty and $\Tau(y,z)=0$ otherwise.\footnote{Note that this could also be inferred from the more general Theorem \ref{thm:LLS-bareLorLength}.}
 
It remains to show that $Y$ is \emph{localizable} (\cite[Def.\ 3.16]{KS:18}), i.e., we need to show that every point $y\in Y$ has an open neighborhood $\Omega$ such that
\begin{itemize}
 \item $L^d(\gamma)\leq C$ for some $C>0$ and all causal curves $\gamma$ contained in $\Omega$,
 \item there is a continuous $\omega\colon\Omega\times\Omega\rightarrow [0,\infty)$ such that $(\Omega,d\rvert_{\Omega\times\Omega},\ll\rvert_{\Omega\times\Omega},\leq\rvert_{\Omega\times\Omega},\omega)$ is a \LpLS with $I^\pm(o)\cap\Omega\neq\emptyset$ for all $o\in\Omega$, and
 \item for all $o,o'\in \Omega$ with $o<o'$ there is a future directed causal curve $\gamma$ from $o$ to $o'$ that is maximal in $\Omega$ and $L(\gamma)=\omega(o,o')\leq\tau(o,o')$.
\end{itemize}

To this end we apply the argument of the proof 
of \cite[Lem.\ 4.3]{GKS:19} to see that we can use $\omega:=\tau\rvert_{U\times U}$ for a suitable neighborhood $U$ of a 
point $y=(t_0,x)\in Y$. Such a suitable neighborhood can be chosen by taking $U$ to be one of the causally convex neighborhoods from Lemma \ref{lem-causally-convex-nbhds} that is contained in the neighborhoods of Lemma \ref{lem-D-com} and Proposition 
\ref{prop-loc-ex-max-cc}. Thus $\omega$ is finite and lower semicontinuous. To see that $\omega$ is also upper semicontinuous 
note that we can adapt the proof of \cite[Thm.\ 3.28]{KS:18} to the simpler local situation in $U$ by using the local 
existence of maximal causal curves (Proposition \ref{prop-loc-ex-max-cc}) and the upper semi-continuity of $L$ (Proposition 
\ref{prop-L-usc}). Moreover, since $U$ is open, one has $I^\pm((t_0,x))\cap U\neq\emptyset$.

This yields that $Y$ is localizable and hence by the above is a \LLSn. It also implies that $Y$ is strongly causal in the sense 
of \cite[Def.\ 2.35(iv)]{KS:18} by using the result for \LLSn s \cite[Thm.\ 3.26(iv)]{KS:18} and Lemma \ref{lem-causally-convex-nbhds}.
\end{pr}

Further, the \LLS $I\times_f X$ is \emph{regular}, i.e., maximal causal curves have a causal character (cf.\ \cite[Def.\ 3.22]{KS:18}). Thus by Proposition \ref{thm-wpd-lls} and Corollary \ref{cor-f-lip} we immediately get the following:
\begin{cor}
  Any \wpd\ $I\times_f X$, where $(X,d)$ is a geodesic locally compact length space, is a strongly causal and regular \LLSn.
\end{cor}

Lemma \ref{lem-jpq} shows that if $X$ is proper the causal diamonds $J(p,q)$ are pre-compact. Moreover, by Lemma \ref{lem-causally-convex-nbhds} any \wpd\ is strongly causal. This is already close to the usual notion of global hyperbolicity. In the next Proposition we will show that \wpds, where $X$ is proper and geodesic, are in fact globally hyperbolic (as defined for \LLSn s in \cite[Def.\ 2.35(v)]{KS:18}).

\begin{prop}
  Let $I\times_f X$ be a \wpd, where $(X,d)$ is a geodesic length space that is a proper metric space. Then $I\times_f X$ is 
globally hyperbolic.
\end{prop}
\begin{pr}
From Theorem \ref{thm-wpd-lls} we know that $Y$ is a strongly causal \LLS and hence non-totally imprisoning by \cite[Thm.\ 
3.26(iii)]{KS:18}. Moreover, 
from Corollary  \ref{cor-J+-closed} we know that $J^\pm(p)$ is closed for every $p\in Y$ and Lemma \ref{lem-jpq} implies that for all $p,q\in Y$ 
the causal diamond $J(p,q)$ is contained in a compact set. Thus $J(p,q)$ is compact and so $Y$ is globally hyperbolic
in the sense of \cite[Def.\ 2.35(v)]{KS:18}.
\end{pr}

As any complete and locally compact length space is proper and geodesic (by the Hopf-Rinow-Cohn-Vossen theorem) we obtain 
the following corollary.
\begin{cor}
 Let $I\times_f X$ be a \wpd, where $X$ is a locally compact, complete length space. Then $I\times_f X$ is globally hyperbolic.
\end{cor}

Recall that a \LpLS is called \emph{geodesic} (\cite[Def.\ 3.27]{KS:18}) if any two causally related points can be joined by 
a maximal causal curve. As any globally hyperbolic \LLS is geodesic (Avez-Seifert, cf.\ \cite[Thm.\ 3.30]{KS:18}), 
we conclude by the above 
that every \wpd\ is geodesic if $X$ is proper and geodesic (in the metric space sense). This implies the following stronger result.
\begin{cor}\label{cor:geodesicifffibergeod}
 Let $X$ be geodesic, then $I\times_f X$ is geodesic. Furthermore, any two timelike related points can be connected 
by a timelike geodesic.
\end{cor}
\begin{proof}
	Let $(x_0,\bar{x}),(y_0, \bar{y})\in Y=I\times_f X$. Because $X$ is geodesic there exists a minimal curve $\beta :[0,d_X(\bar x,\bar y)]\to X$, parametrized by arc-length,  from $\bar x$ to $\bar y$. Let $X'=[0,d_X(\bar x,\bar y)]$ (with the standard metric) and $Y'=I\times_f X'$. Since $X'$ is proper and geodesic, $Y'$ is geodesic and there exists a maximizing curve $\gamma'=(\alpha',\beta'): [0,d_X(\bar x,\bar y)]\to Y'$, with $\beta' $ parametrized by arc-length, from $(x_0,0)$ to $(y_0, d_X(\bar x,\bar y))$. Then $\gamma :=(\alpha',\beta)$ is maximizing from $(x_0,\bar{x})$ to $(y_0, \bar{y})$ in $Y$ by 
	Theorem \ref{thm-structure-of-geod} (ii). The second claim follows from Corollary \ref{cor-f-lip}.
\end{proof}

\section{Curvature bounds}\label{sec-cb}
In this Section we generalize the results of Section \ref{sec-con} to \wpds, i.e., we relate (metric) curvature bounds of the fiber $X$ to timelike curvature bounds of the \wpd~$Y=I\times_f X$, and vice versa.

\begin{lem} \label{lem:taudownifdup} Let $(X,d)$ and $(X',d')$ be two geodesic length spaces. Let $Y:=I\times_f X$ and $Y':=I\times_f X'$. 
	Then for any two pairs of points $x=(x_0,\bar{x}),y=(y_0,\bar{y})\in Y$ and $x'=(x_0,\bar{x}'),y'=(y_0,\bar{y}')\in Y'$ with $d_X(\bx,\by)\geq 
	d'_{X'}(\bx ',\by ')$ one has $\tau(x,y)\leq \tau'(x',y')$.
\end{lem}
\begin{proof}
	If $\tau(x,y)=0$ this obviously holds, so assume $x\ll y$. Let $\beta :[a,b]\to X $ be a minimizing unit-speed geodesic from $\bx$ to $\by$ in $X$. Then by Corollary \ref{cor:geodesicifffibergeod}, there exists a timelike curve $\gamma\equiv (\alpha,\beta):[a,b]\to Y$ from $x$ to $y$ with $L(\gamma)= \tau(x,y)$. Further, let $\beta':[a,b]\to X'$ be a curve from $\bx'$ to $\by'$ in $X'$ such that $L_{X'}(\beta')= d'_{X'}(\bx',\by') $ and $v_{\beta'}$ is constant. Then $L(\beta)=d(\bx,\by)\geq d'_{X'}(\bx',\by') =L(\beta')$ implies $v_{\beta'}\leq v_\beta=1$, so the curve $\gamma':=(\alpha,\beta'):[a,b]\to Y'$ is timelike  and 
	\begin{align*}
	\tau(x,y)&=L(\gamma)=\int_a^b\sqrt{\dot{\alpha}^2-(f\circ\alpha)^2}\leq \int_a^b\sqrt{\dot{\alpha}^2-(f\circ\alpha)^2v_{\beta'}^2}\\
	&=L(\gamma')\leq \tau'(x',y').
	\end{align*}
\end{proof}

Moreover, for causally related points also strict inequalities are preserved in the above Lemma, i.e., if $d_X(\bx,\by)< d'_{X'}(\bx ',\by ')$ then 
$\tau(x,y)<\tau'(x',y')$. From this one obtains immediately the following converse:
\begin{lem}\label{lem-tau-vs-d-conv}
	Let $(X,d)$ and $(X',d')$ be two geodesic length spaces. Let $Y:=I\times_f X$ and $Y':=I\times_f X'$. Then for any two pairs of causally 
	related points $x=(x_0,\bar{x}),y=(y_0,\bar{y})\in Y$ and $x'=(x_0,\bar{x}'),y'=(y_0,\bar{y}')\in Y'$ with $\tau(x,y)\leq \tau'(x',y')$ one 
	has $d_{X'}(\bx',\by')\leq d_X(\bx,\by)$.
\end{lem}

We now turn to the relation between (metric) curvature bounds in the fiber $X$ (in the sense of \cite[Def.~4.6.2]{BBI:01}, cf.\ also Subsection \ref{subsec-ale})
and timelike curvature bounds in the \wpd\ $I\times_f X$ as defined in \cite[Def.~4.7]{KS:18}, cf.\ also Subsection \ref{subsec-lls}.

\begin{thm}\label{thm-X-cb-Y-cb}
	Let $K, K'\in \R$ and let $(X,d)$ be a geodesic length space with curvature bounded below/above by $K$. 
	Then $Y=I\times_fX$ has timelike curvature bounded below/above by $K'$ if $I\times_f \mathbb{M}^2(K)$ has timelike curvature bounded below/above by $K'$.
\end{thm}
\begin{proof}
	As in the proof of Theorem~\ref{thm-wpd-lls}, for any $w\in Y$ we can choose a causally convex neighborhood $U\subseteq Y$ 
	according to 
	Lemma~\ref{lem-causally-convex-nbhds} such that $\tau|_{U\times U}$ is continuous and any two points $x,y\in U$ with $x\ll y$ 
	can be connected by a maximal future-directed timelike curve $\gamma $ in $U$ with $L(\gamma)=\tau(x,y)$. We may further 
	assume that $U$ was chosen small enough to satisfy the following conditions:
	\begin{itemize}
		\item[(i)] There is an open set $V\subseteq X$ on which triangle comparison with 
		$\mathbb{M}^2(K)$ holds and such that for all 
		$\bx \in X$ for which there exists $x_0 \in \R$ such that $(x_0,\bx) \in U$ we have $\bx \in V$.
		\item[(ii)] $U\subseteq [u_0,u_1]\times B^{d_X}_\eps(\bar{w})$, where $\eps $ and $|u_0-u_1|$ are so small that, for 
		some (fixed) $\bar{w}'\in \mathbb{M}^2(K)$ we have that $[u_0,u_1]\times B^{d_{\mathbb{M}^2(K)}}_{2\eps}(\bar{w}')\subseteq 
		I\times_f\mathbb{M}^2(K)$ is contained in a neighborhood $U'$ on which timelike triangle comparison with $\mathbb{L}^2(K')$ 
		holds.
	\end{itemize}
	
	Let $\Delta=(x,y,z)$ be a timelike geodesic triangle in $U$, realized by maximal timelike curves 
	$\gamma_{xy},\gamma_{yz},\gamma_{xz}$ whose side lengths $a,b,c$ satisfy timelike size bounds for $K'$, i.e. $c\geq a+b$ and 
	if $c=a+b$ and $K'>0$ or $c>a+b$ and $K'<0$ then $c<\frac{\pi}{\sqrt{|K'|}}$.
	
	To establish that $Y$ has timelike curvature bounded below by $K'$ we have to show that if $\Delta''=(x'',y'',z'')$ is a comparison triangle of $\Delta=(x,y,z)$ in $\mathbb{L}^2(K')$, then for all points $p,q$ on the sides of $\Delta=(x,y,z)$ and corresponding points $p'',q''$ on the sides of $\Delta''=(x'',y'',z'')$, we have $\tau(p,q)\leq \tau_{\mathbb{L}^2(K')}(p'',q'')$. (To show that $Y$ has timelike curvature bounded above by $K'$ we have to show that $\tau(p,q)\geq \tau_{\mathbb{L}^2(K')}(p'',q'')$.)
	
	We do this in two steps. First, we construct a comparison triangle $\Delta'=(x',y',z')$ in $Y':=I\times_f \mathbb{M}^2(K)$ with $\tau(p,q)\leq \tau_{Y'}(p',q')$ (respectively $\geq$ in case of curvature bounded above).

	The projection $\bar{\Delta}=(\bx,\by,\bz)$ of $\Delta=(x,y,z)$ onto $X$ is a geodesic triangle $\bar{\Delta}$ in $X$ (which can be degenerate) and if $\gamma_{xy}=(\alpha_{xy},\beta_{xy}), \gamma_{yz}=(\alpha_{yz},\beta_{yz})$ and $\gamma_{xz}=(\alpha_{xz},\beta_{xz})$ are the sides of $\Delta=(x,y,z)$, then 
	the sides of $\bar{\Delta}=(\bx,\by,\bz)$ are the minimizing curves $\beta_{xy},\beta_{yz},\beta_{xz}$ and are contained in 
	$V$. Because $V$ was a neighborhood in $X$ on which triangle comparison with $\mathbb{M}^2(K)$ holds, there exists a triangle 
	$\bar{\Delta}'=(\bx',\by',\bz')$ in $\mathbb{M}^2(K)$ such that $d_X(\bar{p},\bar{q})\geq
	d_{\mathbb{M}^2(K)}(\bar{p}',\bar{q}')$ (respectively 
	$\leq$ in case of curvature bounded above). This triangle $\bar{\Delta}'$ in $\mathbb{M}^2(K)$ can be lifted to a triangle 
	$\Delta'=(x',y',z')$ in $Y'=I\times_f \mathbb{M}^2(K)$ given by $x':=(x_0,\bx'),y':=(y_0,\by'),z':=(z_0,\bz')$. By fiber 
	independence (Theorem \ref{thm-structure-of-geod},\ref{thm-structure-of-geod-fib-ind}) $\Delta'$ is a triangle with the same side lengths as $\Delta$ 
	and the points $p'$ and $q'$ corresponding to $p$ and $q$ are 
	exactly $(p_0,\bar{p}')$ and $(q_0,\bar{q}')$. Thus, by Lemma \ref{lem:taudownifdup} $\tau(p,q)\leq \tau_{Y'}(p',q')$ (respectively $\geq$ in case of curvature bounded above).
	
	Now because $\Delta$ satisfies timelike size bounds, the same is true for $\Delta'$. Further, because of the symmetries of 
	$\mathbb{M}^2(K)$ we may additionally suppose that our comparison triangle $\Delta'$ was chosen such that $\bx'=\bar{w}'$, hence 
	$\Delta'\subseteq U'$ by our choice of $U$ (cf.~item (ii)). So, by the timelike curvature bound of 
	$Y'=I\times_f\mathbb{M}^2(K)$ there exists a timelike comparison triangle $\Delta''=(x'',y'',z'')$ of $(x',y',z')$ in 
	$\mathbb{L}^2(K')$ such that for the points $p'',q''$ on $ \Delta''$ corresponding to $p',q'$, we have $\tau_{Y'}(p',q')\leq 
	\tau_{\mathbb{L}^2(K')}(p'',q'')$ (respectively, $\tau_{Y'}(p',q')\geq \tau_{\mathbb{L}^2(K')}(p'',q'')$ for a timelike curvature 
	bound from above). Moreover, by construction and fiber independence (cf.\ 
	Theorem \ref{thm-structure-of-geod}) $\Delta''$ must be a comparison triangle for $\Delta$.
	
	Together with the inequality $\tau(p,q)\leq \tau_{Y'}(p',q')$ (respectively $\geq$) from before we get $\tau(p,q)\leq 
	\tau_{\mathbb{L}^2(K')}(p'',q'')$ (respectively $\geq$), concluding the proof.
\end{proof}

In the case of a smooth warping function $f$ we can give sufficient conditions so that $Y=I\times_f X$ has timelike curvature bounded by $K'$.

\begin{cor}\label{cor-f-smo-cb}
	Let $f\colon I\rightarrow (0,\infty)$ be smooth and let $Y=I\times_f X$, $Y'=I\times_f \mathbb{M}^2(K)$. If $f$ is 
	$K'$-concave (convex), i.e., $f''-K'f\leq 0$ ($f''-K'f\geq 0$) and $K=\sup (K' f^2-(f')^2)$ ($K=\inf (K' f^2 -(f')^2)$) and 
	$X$ has curvature bounded below (above) by $K$, then $Y$ has timelike curvature bounded below (above) by $K'$.
\end{cor}
\begin{pr}
	This follows directly from Theorem \ref{thm-X-cb-Y-cb} and \cite[Prop.\ 7.1]{AB:08} in the special case that the base is one-dimensional and from the fact 
	that sectional curvature bounds in the sense of \cite{AB:08} imply timelike curvature bounds in the sense of \LLSn s
	(cf.\ \cite[Ex.\ 4.9]{KS:18}).
\end{pr}

We apply the preceding corollary to specific spaces and warping functions to obtain:

\begin{cor}
	Let $X$ be a geodesic length space with curvature bounded below/above by $K$ (third column). With the interval $I$ given in the first column and the warping function given in the second column, $I\times_f X$ has timelike curvature bounded below/above by $K'$ (forth column):
	
	\begin{table}[h!]
\begin{tabular}{|c|c|c|c|}
\hline
\textbf{$I$}                       & \textbf{$f$}    & \textbf{$X$ CB b/a by $K$} & \textbf{$I\times_f X$ TLCB b/a by $K'$} \\ \hline
$(0,\pi)$                        & $\sin$        & $-1$                 & $-1$                \\ \hline
$(-\frac{\pi}{2},\frac{\pi}{2})$ & $\cos$        & $-1$                 & $-1$                \\ \hline
$(0,\infty)$                     & $\mathrm{id}$ & $-1$                 & $\hphantom{-}0$                 \\ \hline
$(0,\infty)$                     & $\sinh$       & $-1$                 & $\hphantom{-}1$                 \\ \hline
$\R$                             & $\exp$        & $\hphantom{-}0$                  & $\hphantom{-}1$                 \\ \hline
$\R$                             & $1$           & $\hphantom{-}0$                  & $\hphantom{-}0$                 \\ \hline
$\R$                             & $\cosh$       & $\hphantom{-}1$                  & $\hphantom{-}1$                 \\ \hline
\end{tabular}
\end{table}
\end{cor}
\bigskip

A result in the converse direction holds as well, showing that if $Y$ and $Y'$ satisfy timelike curvature bounds 
then $X$ has a curvature bound. To show this, we first need the following Lemma that establishes that we can lift a geodesic 
triangle $\tilde\Delta$ in $X$ to a timelike geodesic triangle in $Y$, provided $\tilde\Delta$ is small enough.

\begin{lem}\label{lem-lif}
	Let $Y=I\times_f X$ be a \wpd, where $X$ is a geodesic length space. Let $(t_0,\bar p_0)\in Y$, then for all neighborhoods $V\subseteq Y$ of $(t_0,\bar p_0)$  
	there is a constant $C>0$ (depending on $f$ and $V$) such that any convex neighborhood $U$ of $\bar p_0$ in $X$ with $\diam(U)\leq C$ has the property that any geodesic triangle in $U$ can be lifted to a timelike geodesic triangle in $V$.
\end{lem}
\begin{pr}
	Let $(t_0,\bar{p}_0)\in Y$ and $V\subseteq Y$ a neighborhood of $(t_0,\bar p_0)$, and set $f(t_0)=:m>0$. Then there is an $\eps>0$ such that for all $t\in[t_0-\eps,t_0+\eps]\subseteq I$ we have $f(t)\leq 2m$ and $[t_0-\eps,t_0+\eps ]\times B_{\frac{\eps}{2\sqrt{2}m}}(\bar p_0)\subseteq V$. Moreover, set $\alpha(t):=t$, $\alpha\colon[t_0-\eps,t_0+\eps]\rightarrow I$ and $t_\pm:=\alpha(t_0\pm\eps)\in I$. Let $U$ be a convex neighborhood in $X$ with $\diam(U)\leq \frac{\eps}{2\sqrt{2} m}=: C$. Then $U\subseteq B_{\frac{\eps}{2\sqrt{2}m}}(\bar p_0)$. Let $\bar{\Delta}=(\bx,\by,\bz)$ be a geodesic triangle in  $U$. Let $\tilde\beta\colon[0,d(\bx,\by)]\rightarrow U$ be a minimizing unit-speed geodesic connecting $\bx$ and $\by$. We now reparametrize $\tilde\beta$ as follows: $\beta(s):=\tilde\beta(\frac{s-t_0+\eps}{\eps} d(\bx,\by))$ for $s\in[t_0-\eps,t_0]$. Clearly, $\beta\colon[t_0-\eps,t_0]\rightarrow U$ is minimizing and $v_\beta=\frac{d(\bx,\by)}{\eps}$ almost everywhere. We establish that $\gamma:= (\alpha,\beta)\colon[t_0-\eps,t_0]\rightarrow V$ is future directed timelike from $x:=(t_-,\bx)$ to $y:=(t_0,\by)$:
	\begin{align}
	-\dot\alpha^2 + (f\circ\alpha)^2 v_\beta^2 = -1 + f(t)^2\frac{d(\bx,\by)^2}{\eps^2} \leq -1 + m^2\frac{4\diam(U)^2}{\eps^2} \leq -\frac{1}{2}\,.
	\end{align}
	Thus we have $(t_-,\bx)\ll (t_0,\by)$ and analogously $(t_0,\by)\ll(t_+,\bz)=:z$. As $Y$ is geodesic by Corollary \ref{cor:geodesicifffibergeod}, there is a maximizing future directed timelike curve from $x$ to $y$, whose projection to $X$ is a minimizing curve in $X$ by Theorem \ref{thm-structure-of-geod} (i). Note that this projection is in general different from $\beta$. However, one can now proceed as in the proof of Corollary \ref{cor:geodesicifffibergeod} and obtain a maximal timelike curve in $Y$ whose projection is $\beta$ or $\tilde\beta$, respectively. Analogously, one can argue the existence of maximizing causal curves from $y$ to $z$ and from $x$ to $z$ for which the $X$-components are the minimizing geodesics from $\by$ to $\bz$ and from $\bx $ to $\bz$. Thus $\Delta:=(x,y,z)$ is a lift of $\bar{\Delta}$ and by construction $\Delta$ and all its sides lie in $V$.
\end{pr}

With this preparation we can show the following:
\begin{thm}\label{thm-Y-cb-X-cb}
	If $X$ is a geodesic length space, $Y=I\times_f X$ has timelike curvature bounded below (above) by $K'$ and $Y'=I\times_f \mathbb{M}^2(K)$ has timelike curvature bounded above (below) by $K'$ then $X$ has curvature bounded below (above) by $K$.
\end{thm}
\begin{pr}
	Fix $\bar{p}_0\in X$ and $t_0\in I$. Let $V\subseteq Y$ be a comparison neighborhood for $(t_0,\bar p_0)$ in $Y$ such that all timelike triangles in $V$ satisfy the size bounds. Let $C>0$ be given by Lemma \ref{lem-lif}, and let $U$ be a convex neighborhood of $\bar p_0$ in $X$ 
	with $\diam(U)\leq C$. Let $\bar \Delta=(\bx,\by,\bz)$ be a geodesic triangle in $U$ satisfying the appropriate size bounds with its sides realized by unit speed geodesics $\beta_{\bar r,\bar s}$ ($r,s\in\{x,y,z\}$). Let $\bar p,\bar q $ be points on $\bar \Delta$, say $\bar p \in \beta_{\bar x,\bar y}$ and $\bar  q\in \beta_{\bar y,\bar z}$. Then by Lemma \ref{lem-lif} we can lift $\bar\Delta$ to a timelike geodesic triangle $\Delta=(x,y,z)$
	in $I\times_f X$, where $x=(x_0,\bx)$, $y=(y_0,\by)$, $z=(z_0,\bz)$, with $x\ll y \ll z$, and whose sides are realized by future directed timelike maximal curves $\gamma_{r,s}=(\alpha_{r_0,s_0},\beta_{\bar r,\bar s})$ ($r,s\in\{x,y,z\})$. Let $\bar\Delta'=(\bx',\by',\bz')$ be a comparison triangle in $\mathbb{M}^2(K)$ of $\bar\Delta$, whose sides are realized by geodesics $\beta_{\bar r',\bar s'}'$ in $\mathbb{M}^2(K)$. Again we lift $\bar\Delta'$ to a timelike geodesic triangle $\Delta'$ in $Y'$ of the same side lengths as $\Delta$. Moreover, let $\Delta''$ be a comparison triangle in $\mathbb{L}^2(K')$ of $\Delta$, which therefore is also a comparison triangle of $\Delta'$. Let $\bar p',\bar q'$ be points corresponding to $\bar p,\bar q$ on $\bar\Delta'$. 
	We can lift these points to points on $p,q\in \Delta$ and $p',q'\in\Delta'$ as follows: Let $t_{\bar p}\in[0,d_X(\bar r,\bar s)]$, then $p:=\gamma_{r,s}(t_p)=(\alpha_{r_0,s_0}(t_{\bar p}),\bar p)$, and analogously for $\bar q$ and $\bar p',\bar q'$. Then let $p'',q''$ be corresponding points on $\Delta''$ with respect to $p,q$, which are also corresponding to $p',q'$.  At this point we can use the curvature bounds to obtain
	\begin{align}
	\tau(p,q)\leq \tau_{\mathbb{L}^2(K')}(p'',q'')\leq \tau'(p',q')\,.
	\end{align}
	Since by construction $p\leq q$ and $p'\leq q'$, Lemma \ref{lem-tau-vs-d-conv} now gives $d_{\mathbb{M}^2(K)}(\bar p',\bar q')\leq d(\bar p,\bar q)$, concluding the proof.
\end{pr}

\section{Synthetic singularity theorems for generalized cones}\label{sec:sin-thm}

\begin{lem} \label{lem:tauinmodel} Let $K'>K$ (resp.~$K'<K$). Then for small enough corresponding timelike triangles $\Delta=(x,y,z)\in \mathbb{L}^2(K)$ with $x\ll y\ll z$ such that $\tau_{\mathbb{L}^2(K)}(x,z)>\tau_{\mathbb{L}^2(K)}(x,y)+\tau_{\mathbb{L}^2(K)}(y,z)$ (i.e., $x,y,z$ don't lie on a single maximizing geodesic) and $\Delta'=(x',y',z')\in \mathbb{L}^2(K')$, $x'\ll y'\ll z'$ and corresponding points $q\in yz\subseteq I^+(x)$, $q'\in y'z' \subseteq I^+(x')$ with $q\neq y,z$ (and consequently $q'\neq y',z'$) we have
	\begin{align}\label{eq:tauinmodel} \tau_{\mathbb{L}^2(K')}(x',q')>\tau_{\mathbb{L}^2(K)}(x,q) \quad \text{(resp.~}<\text{)}
	\end{align}
\end{lem}

\begin{proof}
	Note first that $K'\geq K$ implies that $\mathbb{L}^2(K)$  has  {\em sectional curvature bounded  below} by $K'$ in the sense of \cite[Eq.~(1.1)]{AB:08} (i.e., $\mathcal{R}\geq K'$ if and only if spacelike sectional curvatures are bounded below by $K'$ and timelike ones bounded above by $K'$) and thus \cite[Thm.~1.1]{AB:08} shows $\tau_{\mathbb{L}^2(K')}(x',q')\geq \tau_{\mathbb{L}^2(K)}(x,q)$. 
	
	To show that strict inequality in the curvatures implies strict inequality of the time separations we follow the proof of \cite[Thm.~1.1]{AB:08}, at times referencing \cite{K:18} for additional details, following the notations therein except that our model space is now  $\mathbb{L}^2(K')$, not $\mathbb{L}^2(K)$, and we use ${}'$ to denote the corresponding objects in the model space (as opposed to $\tilde{}$~in \cite{AB:08, K:18}): 
	
	In this proof the curvature inequality is first used in Corollary 4.5, where the authors use ${R}_{\dot{{\sigma}}}\geq R'_{\dot{\sigma}'}$ and a general comparison result for modified shape operators along a geodesic, see \cite[Thm.~4.3]{AB:08}, to get that the modified shape operators satisfy $S\leq S'$ at corresponding
	points of corresponding non-null geodesic segments $\sigma $ and $\sigma '$. Note that by the proof of \cite[Cor.~4.5]{AB:08}, see also \cite[Lem.~5.2.2]{K:18} for a more detailed argument, $S$ and  $S'$ split into direct sums 
	\begin{equation}\label{eq-direct-sum}
	S= S|_{\dot{\sigma}}\oplus S|_{\dot{\sigma }^\perp}, \qquad S'=S'|_{\dot{\sigma}'}\oplus S'|_{(\dot{\sigma }') ^\perp},
	\end{equation}
	where the summands $S|_{\dot{\sigma}},S'|_{\dot{\sigma}'}$ act on the 1-dimensional spaces tangent to the geodesics $\sigma, \sigma '$ whereas $S|_{\dot{\sigma }^\perp},S'|_{(\dot{\sigma }') ^\perp}$ act on their orthogonal complements. According to \cite[Cor.~4.5]{AB:08} the summands acting on the 1-dimensional spaces tangent to the radial geodesics are actually equal, i.e., $S|_{\dot{\sigma}}=S'|_{\dot{\sigma}'} $ 
	at corresponding points. Because of this it is not possible to get the strict inequality $S<S'$ even if ${R}_{\dot{\sigma}}> R'_{\dot{\sigma}'}$. However, the rigidity part of \cite[Thm.~4.3]{AB:08} (resp.~\cite[Thm.~2.2.3]{K:18}) together with the proof of \cite[Cor.~4.5]{AB:08} {\emph{does}} imply that we get the strict inequality
	\begin{equation}\label{eq-str-ine-S}
	S|_{\dot{\sigma }^\perp}<S'|_{(\dot{\sigma }') ^\perp}\,.
	\end{equation}
	Indeed, if we had equality of these (restricted) shape operators at any $\sigma(t_0)$, $\sigma'(t_0)$ then \cite[Thm.~4.3]{AB:08} would imply that ${R}_{\dot{\sigma}(t)}=R'_{\dot{\sigma}'(t)}$ for all $t\leq t_0$. However, this cannot hold as $K'>K$ implies that ${R}_{\dot{\sigma}}> R'_{\dot{\sigma}'}$ for all parameter values.
	
	It remains to argue that the strict inequality \eqref{eq-str-ine-S} carries through to a strict inequality in the time separations as long as the three points don't lie on a common geodesic segment. Thus, let $x,y,z$ be as in the statement of the lemma and let $\gamma \equiv \gamma_{zy}:[0,1]\to \mathbb{L}^2(K)$ be the geodesic from $y$ to $z$. For $q(s)=\gamma(s)$, let $\gamma_{q(s)x}:[0,1]\to \mathbb{L}^2(K)$ be the geodesic from $x$ to $q(s)$ and for each $s$ denote its tangent vector at the endpoint by  $w(q(s))$, i.e. $w(q(s))=\dot{\gamma}_{q(s)x}(1)$. Then, because $x,y,z$ don't lie on a single geodesic, $ \dot{\gamma}(s)\not \parallel w(q(s))$ and hence, replacing \cite[Cor.\ 4.5]{AB:08} by \eqref{eq-str-ine-S}  in the proof of \cite[Cor.\ 4.6]{AB:08} (and keeping in mind the direct sum decompositions \eqref{eq-direct-sum} of $S,S'$ with $S|_{\dot{\sigma}}=S'|_{\dot{\sigma}'} $), we get $ \langle S_{\gamma_{q(s)x}}(\gamma(s))\dot{\gamma}(s),\dot{\gamma}(s)\rangle < (1-K' h_{K',x}(\gamma(s)))\langle \dot{\gamma}(s), \dot{\gamma}(s)\rangle$. Thus, following \cite[Prop.~5.2]{AB:08}, $(h_{K',x}\circ\gamma )\ddot{\;}\, (s)< (1 -K' h_{K',x}(\gamma(s))) \langle \dot{\gamma}(s), \dot{\gamma}(s) \rangle $.
	
	Defining $f:=h_{K',x}\circ \gamma_{yz} - h'_{K',x'}\circ {\gamma}'_{y'z'} $ as in \cite[Prop.~5.2]{AB:08}, we get  $\ddot{f}+ E(\gamma) K f<0$ as contrasted with \cite[Prop.~5.2]{AB:08}. From this $f\geq 0$ follows as in \cite[Prop.~5.2]{AB:08}, see \cite[Lem.~5.1.1]{K:18} for details. We now argue that $f>0$ on the open interval $(0,1)$. Assume there exists $s_0\in (0,1)$ with $f(s_0)=0$. Then $f$ has a local minimum at $s_0$, so $\ddot{f}(s_0)\geq 0$, contradicting the strict inequality $\ddot{f}(s_0)<E(\gamma) K f(s_0)=0$.
	
	Thus, $h_{K',x}\circ \gamma_{yz} (s)> h'_{K',x'}\circ {\gamma}'_{y'z'}(s)$ for all $s\in (0,1)$, and in particular $h_{K',x}(q)>h'_{K',x'}(q')$. Since $h_{K',x}(q)$ is really only a function of the \emph{signed} distance $-\tau_{\mathbb{L}^2(K)}(x,q)$ and  $h'_{K',x'}(q')$ is the {\em same} function of the \emph{signed} distance $-\tau_{\mathbb{L}^2(K')}(x',q')$ and this function is strictly increasing (see \cite[Lem.~4.1.4]{K:18} or simply by the definition of the $h_{K,x}$ in \cite[Eq.~(1.6)]{AB:08}), this implies $-\tau_{\mathbb{L}^2(K)}(x,q)>-\tau_{\mathbb{L}^2(K')}(x',q')$, i.e.,
	\[\tau_{\mathbb{L}^2(K)}(x,q)<\tau_{\mathbb{L}^2(K')}(x',q') \]
	and we are done with the case $K'>K$.
	
	The $K'<K$ case follows from this immediately by simply switching $K$ and $K'$ above.		
\end{proof}

\begin{thm} \label{thm:singthm}
	Let $X$ be a geodesic length space, $Y=I\times_f X$ with $f\colon I\to (0,\infty)$ smooth. Assume that $Y$ has timelike curvature bounded below (above) by $K$, then $f$ is 
	$K$-concave (convex), i.e., $f''-Kf\leq 0$ ($f''-Kf\geq 0$).
\end{thm}

\begin{proof} 
	We only treat the case where $Y$ has timelike curvature bounded below by $K$. The proof for a bound from above is analogous.

	Assume to the contrary that $f$ is not $K$-concave, i.e., there exists $t_0\in I$ such that $f''(t_0)> K f(t_0)$. Then there exists an interval $J\subseteq I$, $t_0\in J$ and $K'>K$ such that $f'' >K' f$ on $J$. Define $Y':=J\times_f \R $, then $Y'$ is a smooth two-dimensional Lorentzian manifold with (timelike) sectional curvature $\mathcal{R}=\frac{f''}{f}>K'$. Since sectional curvature bounds in the sense of \cite[Eq.~(1.1)]{AB:08} (i.e., $\mathcal{R}\geq K'$ if and only if spacelike sectional curvatures are bounded below by $K'$ and timelike ones bounded above by $K'$) imply timelike curvature bounds in the sense of \LLSn s
	(cf.~\cite[Ex.\ 4.9]{KS:18}, based on \cite[Thm.~1.1]{AB:08}), we get that  $Y'$ has timelike curvature bounded above by $K'>K$. Let $x'=(x_0,0), y'=(y_0,\by'),z'=(z_0,0)$ be points in $Y'$ forming a small timelike geodesic triangle. Let $p'=x'$ and $q'=(q_0,\bar{q}')\in y'z'$ and let $p'',q''$ and $p''',q'''$ be corresponding points for $p',q'$ on the sides of the comparison triangle for $\Delta(x',y',z')$ in $\mathbb{L}^2(K')$ and $\mathbb{L}^2(K)$, respectively. Then
	\[\tau_{Y'}(p',q')\geq \tau_{\mathbb{L}^2(K')}(p'',q'')>\tau_{\mathbb{L}^2(K)}(p''',q'''),\]
	where we used Lemma~\ref{lem:tauinmodel} to get the last strict inequality.
	
	Now let $\bx \in X$ be arbitrary and set $x:=(x_0,\bx ),y:=(y_0,\by ),z:=(z_0,\bx )$, where $\by \in X$ is chosen such that $d_X(\bx,\by)=d_{\R}(0,\by')=|\by'|$. Then by fiber independence (Theorem \ref{thm-structure-of-geod},\ref{thm-structure-of-geod-fib-ind}), $\Delta =(x,y,z)$ is a triangle in $Y$ corresponding to $(x',y',z')$ in $Y'$. Let $p=x$ and let $q$ be the point corresponding to $q'$ on the side $yz$ of $\Delta $. Then, again by fiber independence, $\tau_Y(p,q)=\tau_{Y'}(p',q')$ and using the assumption that $Y$ has timelike curvature bounded below by $K$ we obtain the contradiction
	\[\tau_{\mathbb{L}^2(K)}(p''',q''') \geq \tau_Y(p,q)=\tau_{Y'}(p',q')>\tau_{\mathbb{L}^2(K)}(p''',q''').\qedhere \]
\end{proof}
\begin{rem}In the future it would also be interesting to examine if one might still obtain $K$-concavity ($K$-convexity) in the barrier sense for non-smooth warping functions $f$.
\end{rem}

We now relate non-positive lower timelike curvature bounds to the existence of singularities, i.e., incomplete causal geodesics. To this end we first recall some notions and results from \cite{GKS:19}. A \emph{geodesic} in a \LLS $X$ is a causal curve that is locally maximizing (\cite[Def.\ 4.1]{GKS:19}) and for a smooth strongly causal spacetime $(M,g)$ one has that causal pregeodesics of $(M,g)$ are geodesics in the synthetic sense above and vice versa (of the same causal character) (\cite[Thm.\ 4.4]{GKS:19}). Moreover, a geodesic $\gamma\colon[a,b)\rightarrow X$ is \emph{extendible} if there is a geodesic $\bar\gamma\colon[a,b]\rightarrow X$ such that $\bar\gamma\rvert_{[a,b)}=\gamma$. Otherwise, $\gamma$ is called \emph{inextendible}. Also, we have an analogous notion of timelike geodesic completeness in the synthetic setting: A \LLS $X$ is said to have \emph{property (TC)} if all inextendible timelike geodesics have infinite $\tau$-length (\cite[Def.\ 5.1]{GKS:19}). This notion is again compatible with the smooth spacetime case as a smooth and strongly causal spacetime is timelike geodesically complete if and only if it has property (TC) (\cite[Lem.\ 5.2]{GKS:19}). Consequently, we call a timelike geodesic \emph{incomplete} if it has finite $\tau$-length and the space $X$ \emph{timelike geodesically incomplete} if there is an inextendible timelike geodesic which is incomplete. Analogous notions are defined for past and future incompleteness.

\begin{cor}\label{cor:singthm} Let $X$ be a geodesic length space, $Y=I\times_f X$ with $I=(a,b),\,f:I\to (0,\infty)$ smooth. Assume that $Y$ has timelike curvature bounded below by $K$. Then:
	\begin{enumerate}
		\item[(i)] If $K<0$, then $a>-\infty$ and $b<\infty$ and hence the time separation function $\tau_Y $ of $Y$ is bounded by $b-a$. Thus any such $Y$ is timelike geodesically incomplete.
		\item[(ii)] If $K=0$ and $f$ is non-constant, then $a>-\infty $ 
		or $b<\infty$ 
		and hence $Y$ is past or future timelike geodesically incomplete.
	\end{enumerate}
\end{cor}

\begin{proof}
We first show (i): Assume first that $b=\infty$. Set $u:=\frac{f'}{f}$. Theorem \ref{thm:singthm} implies $f''\leq Kf$, so $u$ satisfies the differential inequality $u'\leq -u^2+K\leq \min\{-u^2,K\}$. Since $K<0$ this shows that there exists some $s_0\in (a,\infty)$ with $u(s_0)<0$. Let $s_\mathrm{max}$ be the maximal $s>s_0$ such that $-\infty <u|_{[s_0,s)}<0$ and let $s_1:=s_0-\frac{1}{u(s_0)}>s_0$. Then integrating $u'\leq -u^2$ shows that $u\leq 
\frac{1}{s-s_1} = \frac{1}{\frac{1}{u(s_0)}+s-s_0}<u(s_0)<0$ on $(s_0,\min\{s_\mathrm{max}, s_1\})$. But from this we see that  $s_\mathrm{max}\leq s_1<\infty$ and $u\to -\infty $ as $s\to s_\mathrm{max}$ (otherwise it would contradict the maximality of $s_\mathrm{max}$). Since $u=\frac{f'}{f}$ this implies $f\to 0$ as $s\to s_\mathrm{max}$, contradicting $f>0$ on $(a,\infty)$.
	
	To show that $a>-\infty$ we simply reverse the time orientation, i.e., we consider $\tilde{f}(s):=f(-s)$ instead of $f$.
	
	For (ii), since $f$ is non-constant we have $f'(s_0)\neq 0$ for some $s_0$. If $f'(s_0)<0$, then $u(s_0)<0$ and we get a contradiction to $b=\infty$ as in (i). If $f'(s_0)>0$, we can again just reverse the time orientation to get $\tilde{f}'(-s_0)<0$ and a contradiction to $a=-\infty$.
\end{proof}

\begin{rem}
	 If $(X,h)$ is a smooth $n$-dimensional Riemannian manifold, then $Y=I\times_f X$ with metric $g=-dt^2+f(t)^2 h$ is a smooth Lorentzian manifold and we may compare Corollary \ref{cor:singthm} with the Hawking singularity theorem and the Lorentzian Myers' theorem (\cite[Thm.~55A and 55B]{ONe:83}, \cite[Thm.~11.9]{BEE:96}) applied to $Y$. A key assumption in both these theorems is a bound on the timelike Ricci curvature, which is implied by certain sectional curvature bounds: In any smooth $(n+1)$-dimensional Lorentzian manifold a bound on timelike sectional curvatures from below/above by $K$ implies a bound on the timelike Ricci curvature from above/below by $-n\,K$. However, even in the smooth setting it is not known, to the best of our knowledge, if a bound on merely the \emph{timelike} sectional curvatures, which is strictly weaker than assuming a sectional curvature bound in the sense of \cite[Eq.~(1.1)]{AB:08}, will imply triangle comparison for \emph{timelike} triangles, i.e., a timelike curvature bound as in \cite{KS:18}, or vice versa. To be more precise, let $(M,g)$ be a smooth Lorentzian manifold that has timelike sectional curvatures bounded below by some $K\in\R$, then it is unclear if $(M,g)$ (viewed as a \LLSn) has timelike curvature bounded below by $K$ in the sense of triangle comparison with $\mathbb{L}^2(K)$. Thus, a timelike curvature bound as in \cite{KS:18} might in general not imply the corresponding timelike Ricci curvature bounds. However, in the specific warped product setting we are considering, the following simple relationship holds: For any $V\in TX$ we have $\mathrm{Ric}(\partial_t,\partial_t)=-n\frac{f''}{f} =-n\mathcal{K}(\partial_t,V)$, where $\mathcal{K}(\partial_t,V)$ denotes the sectional curvature of the timelike plane spanned by $\partial_t$ and $V$. So Theorem \ref{thm:singthm} shows that triangle comparison for \emph{timelike} triangles implies both a bound on the sectional curvatures of all timelike planes orthogonal to $X$ and on $\mathrm{Ric}(\partial_t,\partial_t)$. More specifically, timelike curvature bounded below by $K$ implies that $\mathrm{Ric}(\partial_t,\partial_t)\geq -nK$. 

Now the comparison with Hawking's theorem is straightforward: It is well known that the key assumptions for Hawking's theorem to hold boil down to $\mathrm{Ric} (\dot{\gamma},\dot{\gamma})\geq 0$ for all timelike unit speed geodesics $\gamma $ starting orthogonally to an initial Cauchy surface $\Sigma$ with mean curvature $H_\Sigma <\beta <0$. In our warped product we may take $\Sigma$ to be of the form $\{t_0\}\times X$. Then the mean curvature $H_\Sigma $ equals $n \frac{f'(t_0)}{f(t_0)}$ and $\dot{\gamma}(t)=\partial_t|_{\gamma(t)}$, so Hawking's singularity theorem corresponds exactly to the $K=0$ case of Corollary \ref{cor:singthm} (and whether one has a past or a future singularity is determined by the sign of  $H_\Sigma= n \frac{f'(t_0)}{f(t_0)}$, which equals the sign of $\frac{f'(t_0)}{f(t_0)}$).
	
In the $K<0$ case, Corollary \ref{cor:singthm} implies that the timelike diameter of $I\times_f X$ is bounded by $b-a<\infty$, so this case corresponds to the Lorentzian Myers' theorem. Note that while the standard formulation of the Lorentzian Myers' theorem (as in \cite{BEE:96}) requires $\mathrm{Ric}(W,W)\geq n\,|K|>0$ for all unit timelike vectors $W$, one can use the techniques of the proof of the Hawking singularity theorem to show that to get a bound on the timelike diameter it is sufficient to assume $\mathrm{Ric} (\dot{\gamma},\dot{\gamma})\geq n\,|K|>0$ for all timelike unit speed geodesics $\gamma $ starting orthogonally to a Cauchy surface $\Sigma$ (cf.~\cite[Thm.~4.2 and Rem.~4.3]{Gra:16}). However, the bound obtained in this way may be larger than $\frac{\pi}{\sqrt{|K|}}$. With this in mind, we see that also the $K<0$ case of Corollary \ref{cor:singthm} corresponds directly to its smooth counterpart.
\end{rem}

Following \cite[Def.\ 12.16]{ONe:83} we define \emph{big bang} and \emph{big crunch} singularities as follows:

\begin{defi}
 Let $Y=(a,b)\times_f X$ be a \wpd, where $f$ is smooth. Then
 \begin{enumerate}
  \item[(i)] the \wpd\ $Y$ has a \emph{big bang singularity} at $a$ if $f(t)\to 0$ and $f'(t)\to\infty$ as $t\searrow a$, and
  \item[(ii)] the \wpd\ $Y$ has a \emph{big crunch singularity} at $b$ if $f(t)\to0$ and $f'(t)\to-\infty$ as $t\nearrow b$.
 \end{enumerate}
\end{defi}

\begin{cor}\label{cor-big-ban} If $Y$ has a big bang or a big crunch singularity, then $Y$ cannot have timelike curvature bounded above by any $K\in \R$.
\end{cor}
\begin{proof} Assume that $f(s)\to 0$ and $f'(s)\to -\infty $ as $s\to b$ and that $Y$ has timelike curvature bounded above by $K$. First, if $K\geq 0$ then $f''\geq  K f \ge 0$, contradicting  $f'\to-\infty$. So let $K<0$ and let $s_0$ be such that $f\leq 1$ on $(s_0,b)$. Then on this interval $f'' \geq K f\geq K$, so $f'(s)\geq K(s-s_0)+f'(s_0)$, contradicting $f'(s)\to -\infty $. The big bang case follows by reversing the time orientation.
\end{proof}

\section{Conclusion and outlook}
By transporting the notion of a (generalized) cone into the synthetic Lorentzian setting we achieved the following objectives:
\begin{itemize}
 \item We established that the causality of generalized cones is optimal, even for fibers that are not necessarily manifolds.
 \item Generalized cones are instances of strongly causal \LLSn s.
 \item Timelike curvature bounds of a generalized cone can be related to the (metric) curvature bounds of its fiber and vice versa.
 \item Our results allow one to generate an abundance of examples of \LLSn s with timelike curvature bounds.
 \item We proved singularity theorems for generalized cones that are direct analogues of the Hawking singularity
 	theorem and the Lorentzian Myers' theorem. That is, we showed how non-positive lower timelike curvature bounds 
  imply the existence of incomplete timelike geodesics.
\end{itemize}
Our methods are also expected to be applicable to more general spaces with one-dimensional fiber, like the Colombini-Spagnolo metrics discussed in Remark \ref{rem-cau-pla}. More generally, topics of further investigation are the following:
\begin{itemize}
 \item Generalize our methods and results to spaces of the form $I\times X$, where on each fiber $\{t\}\times X$ one has a $t$-dependent metric. This would generalize the metric splitting for globally hyperbolic spacetimes.
 \item Extend the results of Corollary \ref{cor-f-smo-cb} to non-smooth warping functions $f$ that are $K'$-convex or $K'$-concave in the support sense (cf.\ \cite{AB:03}).
 \item Put a Lorentzian structure on general warped products $B\times_f F$, where the base $B$ is a \LLS and the fiber $F$ is a metric (length) space.
 \item Extend the singularity theorems of Section \ref{sec:sin-thm} to lower regularity (i.e., lower the regularity of the warping function $f$).
 \item Relate our approach to the theory of Sormani and Vega \cite{SV:16}, in particular to the results on 
warped products of Allen and Burtscher \cite{AB:19}.
\end{itemize}

\bigskip

\noindent
{\bf Acknowledgments} 
The authors are grateful to the anonymous referees for their detailed feedback and careful reading of the article.

This work was supported by research grants P28770 and J4305 of the Austrian Science Fund FWF and parts of this work were carried out while Melanie Graf was at the
University of T\"ubingen.

\appendix
\section{Appendix: Lorentzian length structures}\label{app}

In this appendix we want to outline that in the situation where one merely has 
\begin{itemize}
 \item a notion of causal curves and
 \item a notion of length of such curves,
\end{itemize}
one can already construct a \LpLS without push-up or topology, hence without lower semicontinuity of $\tau$, 
but such that nevertheless the time separation function is intrinsic. So in a certain sense the situation is even better 
than starting out with a \LpLSn. In doing so we can reproduce some of the results of Section \ref{sec:gen_cone} but in greater generality. Also, this fact was already mentioned in \cite{KS:18}, so here we expand on \cite[Rem.\ 5.11(ii)]{KS:18} and provide some arguments to illustrate the matter.
\medskip

In analogy with the metric case (cf.\ \cite[Sec.\ 2.1]{BBI:01}) we first define a \emph{Lorentzian length structure}. To this end we also need to introduce notions of admissible curves and length functionals.

\begin{defi}
 Let $(X,d)$ be a metric space and denote by $\A$ the class of absolutely continuous curves from an interval into $X$. Let $\I^\pm\subseteq \Con^\pm\subseteq \A$ be four families of of absolutely continuous curves. Then $(\I^\pm,\Con^\pm)$ is 
called \emph{admissible} if it satisfies the following axioms.
\begin{description}
\item [(C1)]\label{C1}
  Every curve in $\Con^\pm$ (hence in $\I^\pm$) is never constant, i.e., restrictions to non-trivial subintervals are non-constant.
\item [(C2)]\label{C2}
  The classes $\I^\pm$ and $\Con^\pm$ are closed under (non-trivial) restrictions, e.g., if $\gamma\colon[a,b]\to X$ is in $\I^\pm$ or 
in $\Con^\pm$, and $a\leq c< d\leq b$ then the restriction $\gamma|_{[c,d]}$ of $\gamma$ to $[c,d]$ is in $\I^\pm$ or $\Con^\pm$, 
respectively.
\item[(C3)]\label{C3}
  The classes $\I^\pm$ and $\Con^\pm$ are closed under concatenations, that is if $\gamma\colon[a,b]\to X$ is a curve such that 
the (non-trivial) restrictions $\gamma|_{[a,c]}$ and $\gamma|_{[c,d]}$ are in $\I^\pm$ or in $\Con^\pm$ for some 
$a\leq c\leq b$ then so is their concatenation $\gamma$.
\item[(C4)]\label{C4} 
  The classes $\I^\pm$ and $\Con^\pm$ are closed under reparametrizations: Let $\gamma\colon J'$ $\to X$ be in $\I^\pm$ or in $\Con^\pm$ and let $\phi\colon J\to J'$ be a strictly increasing continuous map defined on an interval $J\subseteq\R$ such that it and its inverse are absolutely continuous. Then $\gamma\circ \phi\in\I^\pm$ or $\in\Con^\pm$, respectively. Moreover, if $\phi$ is as above but orientation reversing, i.e., strictly decreasing, then $\gamma\circ \phi\in I^\mp$ or $\in \Con^\mp$, respectively.
\end{description}
We call the curves in $\I^\pm$ \emph{future/past directed timelike} curves and the ones in $\Con^\pm$ \emph{future/past directed causal} curves. Moreover, set $\I:=\I^+\cup \I^-$ and $\Con:=\Con^+\cup\Con^-$.
\end{defi}

\begin{defi}{Lorentzian length structure (LLStr)}\label{def-LLStr}
A \emph{Lorentzian length structure} on a metric space $(X,d)$ is an admissible tuple $(\I^\pm,\Con^\pm)$ of subsets of $\A$ 
together with a function
\begin{equation}
L\colon \Con\ \to\ [0,\infty], 
\end{equation}
called the \emph{Lorentzian length functional}, which satisfy the following list of properties:
\begin{description}
\item[(L1)]\label{L1}
  $L$ \emph{is additive}: If $\gamma\colon[a,b]\to X$ is in $\Con$ then $L(\gamma)=L(\gamma|_{[a,c]})+L(\gamma|_{[c,d]})$ for 
any $c\in(a,b)$.
\item[(L2)]\label{L2}
  $L$ is \emph{invariant under reparametrizations}, i.e., for $\gamma$ and $\phi$ as in (C4) we require 
$L(\gamma\circ\phi)=L(\gamma)$.
\item[(L3)]\label{L3}
 For every $\gamma\in\I$ we have $L(\gamma)>0$.
\item[(L4)]\label{L4}
The \emph{length structure respects the topology of $X$} in the following
  sense:
    $L$ depends continuously on the parameter of the curve, e.g., if $\gamma\colon[a,b]\to X$ is in $\Con$ we 
    set $L(\gamma,a,t):=L(\gamma|_{[a,t]})$. Then $t\mapsto L(\gamma,a,t)$ is continuous.
\end{description}
We write $(X,d,\I^\pm,\Con^\pm,L)$ for a \lstrn.
\end{defi}

Recall that a \emph{causal space} $(X,\ll,\leq)$ on a set $X$ is constituted by two transitive relations $\ll$ and $\leq$ on 
$X$, where $\leq$ is additionally reflexive and contains $\ll$, cf.\ \cite[Def.\ 2.1]{KS:18} (this is a slightly stronger 
notion than the one introduced in \cite{KP:67}).

\begin{defi}\label{def-lst-cr}
  Let $(X,d,\I^\pm,\Con^\pm,L)$ be a \lstrn. For $x,y\in X$ define $x\leq y$ if $x=y$ or there exists a $\gamma\in\Con^+$ from $x$ to $y$. Moreover, define $x<y$ if $x\leq y$ and $x\neq y$. Also, define $x\ll y$ if there exists a $\gamma\in\I^+$ from $x$ to 
$y$.
\end{defi}

\begin{lem}\label{lem-cau-spa}
 Let $(X,d,\I^\pm,\Con^\pm,L)$ be a \lstr and let the relations $\ll$, $\leq$ be given by Definition \ref{def-lst-cr}. Then $(X,\ll,\leq)$ 
is a causal space.
\end{lem}
\begin{pr}
 This follows from axiom \descref{C3} and $\I^+\subseteq\Con^+$.
\end{pr}

\begin{defi}\label{def-lst-tau}
  Let $(X,d,\I^\pm,\Con^\pm,L)$ be a \lstrn. For $x,y\in X$ define the \emph{time separation} of $x$ and $y$ by
\begin{equation}
 \tau(x,y):=\sup\{L(\gamma):\gamma\in\Con^+\text{ connects } x \text{ and } y\}\,,
\end{equation}
if the set of connecting future directed causal curves is non-empty, otherwise set $\tau(x,y):=0$.
\end{defi}

The time separation function $\tau$ has the following properties.
\begin{lem}\label{lem-ele-pro-tau}
  Let $(X,d,\I^\pm,\Con^\pm,L)$ be a \lstr and let the time separation function $\tau$ and the causal relations $\ll,\leq$ be as defined 
above. Then
\begin{enumerate}
 \item[(i)]\label{lem-ele-pro-tau-not} $\tau(x,y)=0$ if $x\not\leq y$ and
 \item[(ii)] $\tau(x,y)>0$ if $x\ll y$.
\end{enumerate}
\end{lem}
\begin{pr}
 \begin{enumerate}
 \item[(i)] This is immediate from the definitions.
 \item[(ii)] Let $x\ll y$, then there is a timelike curve $\gamma\in\I^+$ from $x$ to $y$ and by axiom \descref{L3} we have 
$L(\gamma)>0$. Thus $0<L(\gamma)\leq \tau(x,y)$.\qedhere
 \end{enumerate}
\end{pr}

The reverse triangle inequality holds:
\begin{lem}\label{lem-rev-tri-ine-llstr}
  Let $(X,d,\I^\pm,\Con^\pm,L)$ be a \lstr with time separation function $\tau$ and causal relations $\ll,\leq$ as defined 
above. Then for all $x,y,z\in X$ with $x\leq y\leq z$
\begin{equation}
 \tau(x,y) + \tau(y,z) \leq \tau(x,z)\,.
\end{equation}
\end{lem}
\begin{pr}
If $x=y$ then either $\tau(x,x)=0$ or there exists a causal curve in $\Con^+$ from $x$ to $x$. In the first case we are done. Analogously, if $y=z$ either $\tau(y,y)=0$ 
or there exists a causal curve in $\Con^+$ from $y$ to $y$. Again in the first case we are done. So we can without loss of generality assume that there are causal curves connecting $x$ to $y$ and connecting $y$ to $z$. Let $\gamma,\lambda\in\Con^+$ 
with $\gamma$ from $x$ to $y$ and $\lambda$ from $y$ to $z$. Then by axiom \descref{C3} the concatenation $\gamma*\lambda$ is in $\Con^+$ and connects $x$ to $z$. Thus $L(\gamma) + L(\lambda) = L(\gamma*\lambda) \leq \tau(x,z)$, by \descref{L1}. 
Taking now the supremum over all future directed causal curves connecting $x$ and $y$ and the ones connecting $y$ to $z$ gives the claim.
\end{pr}

To investigate only the algebraic properties and further consequences of having a set with a \lstr we relax the notion of 
\LpLS as introduced in \cite[Def. 2.8]{KS:18} and call the resulting generalization a \emph{bare} \LpLSn. In particular, we do not 
want 
to assume semi-continuity properties of the time separation function from the beginning.

\begin{defi}\label{def:bare}
Let $(X, \ll, \leq)$ be a causal space together with a function  $\tau\colon X\times X\rightarrow [0,\infty]$ that 
satisfy the following properties:
\begin{enumerate}
 \item[(i)] If $x\ll y$ then $\tau(x,y)>0$.
 \item[(ii)] The reverse triangle inequality holds: for all $x,y,z\in X$ with $x\leq y\leq z$ one has
\begin{equation}
 \tau(x,y)+\tau(y,z)\leq \tau(x,z)\,.
\end{equation}
\end{enumerate}
Then $(X,\ll,\leq,\tau)$ is called a \emph{bare \LpLSn}.
\end{defi}

This is a generalization of the notion of a \LpLSn, where we do not assume that $X$ comes with a metric (topology), push-up 
to hold or that $\tau$ is lower semicontinuous with respect to this topology \footnote{Of course, by endowing $X$ with the discrete metric, any function from $X$ to another topological space is continuous. Thus any such bare \LpLS without topology 
can be viewed as a \LpLS in the sense of \cite[Def.\ 2.8]{KS:18} (except that push-up need not hold). However, note that then there are no non-constant curves, hence no causal curves.}. 
\bigskip

Given a bare \LpLS and a metric $d$ on $X$ one can still define causal curves, their length etc., and get the same basic properties (following the steps as in \cite[Sec.\ 2]{KS:18}). For example, given now the causal relations defined in Definition \ref{def-lst-cr} a $\leq$-causal curve is a locally Lipschitz curve $\gamma$ such that $\gamma(s)\leq\gamma(t)$ for any two parameter values with $s\leq t$ (cf.\ \cite[Def.\ 2.18]{KS:18}). Note that any curve in $\Con^+$ is $\leq$-causal, and analogously for timelike curves.

Finally, we define
\begin{equation}
 \Tau(x,y):=\sup\{L_\tau(\gamma): \gamma \text{ future directed } \leq\text{-causal from } x \text{ to }y\}\,,
\end{equation}
if the set of connecting future directed $\leq$-causal curves from $x$ to $y$ is non-empty. Otherwise, we set $\Tau(x,y):=0$.

\begin{defi}
 A bare \LpLS $(X,\ll,\leq,\tau)$ is called \emph{bare \LLSn} if $\Tau=\tau$, i.e., if the time separation function $\tau$ is intrinsic (for some background metric $d$ on $X$). 
\end{defi}
\medskip

At this point we are able to establish that a \lstr gives rise to a bare \LLSn.
\begin{thm}\label{thm:LLS-bareLorLength}
 Let $(X,d,\I^\pm,\Con^\pm,L)$ be a \lstr and define the causal relations $\leq$, $\ll$ and the time separation function $\tau$ as 
above. Then $(X,\ll,\leq,\tau)$ is a bare \LLSn.
\end{thm}
\begin{pr}
That $(X,\ll,\leq,\tau)$ is a bare \LpLS follows from Lemmata \ref{lem-cau-spa}, \ref{lem-ele-pro-tau} and 
\ref{lem-rev-tri-ine-llstr}.

 As noted, any $\Con^+$-causal curve is $\leq$-causal and analogously for $\I^+$ and $\ll$. Moreover, for any $\leq$-causal curve $\gamma$ (hence also for any $\gamma\in\Con^+$) from $x$ to $y$ one has $L_\tau(\gamma)\leq \tau(x,y)$ by the definition of the $\tau$-length. Let $(\gamma\colon[a,b]\rightarrow X)\in\Con^+$ and let $a=t_0<t_1<\cdots < t_N = b$ be a partition of $[a,b]$, 
then by \descref{L1}
 \begin{equation}
L(\gamma) = \sum_{i=0}^{N-1} L(\gamma\rvert_{[t_i,t_{i+1}]}) \leq \sum_{i=0}^{N-1} \tau(\gamma(t_i),\gamma(t_{i+1}))\,.
 \end{equation}
Taking the infimum over all partitions of $[a,b]$ yields that $L(\gamma)\leq L_\tau(\gamma)$ for any $\gamma\in\Con^+$ 
(defined on a compact interval). From this we immediately get $\tau \leq \Tau$ from the definitions by first taking the supremum on the right hand side and then on the left hand side.

 It remains to show that $\Tau \leq \tau$. If $\Tau(x,y)=0$ there is nothing to do, so let 
$\Tau(x,y)>0$. This means that for any $\eps>0$ there is a future directed $\leq$-causal curve $\gamma$ from $x$ to $y$ with 
$\Tau(x,y)<L_\tau(\gamma) + \eps \leq \tau(x,y)+\eps$. Here we used that $L_\tau(\gamma)\leq \tau(x,y)$, and since $\eps>0$ 
was arbitrary we are done.
\end{pr}

\bibliographystyle{alphaabbr}
\bibliography{LLSEx}
\addcontentsline{toc}{section}{References}

\end{document}